\documentclass[12pt]{article}

\usepackage{amsmath}
\usepackage{amssymb}
\usepackage{amsthm}
\usepackage{color}
\usepackage[dvipsnames]{xcolor}
\usepackage{url}
\usepackage[top=1.5cm, 	bottom=2cm, left=3cm, right=3cm]{geometry}
\usepackage{hyperref}
\hypersetup{colorlinks=false, linkbordercolor={1 1 1}, citebordercolor={1 1 1}, urlbordercolor={1 1 1}} 
\usepackage[normalem]{ulem}
\reversemarginpar   

\usepackage{caption}
\usepackage{graphicx}
\graphicspath{{Figures/}} 
\usepackage{mathrsfs}

\numberwithin{equation}{section}

\usepackage{subcaption}
\usepackage{tikz}

\tikzset{dot/.style={circle, fill, inner sep=2pt}}

\newcommand{\mapsfrom}{\mathrel{\reflectbox{\ensuremath{\longmapsto}}}}

\usepackage{mathtools}
\usepackage[dvipsnames]{xcolor}

\newtheorem{theorem}{Theorem}
\newtheorem{proposition}[theorem]{Proposition}
\newtheorem{lemma}[theorem]{Lemma}
\newtheorem{corollary}[theorem]{Corollary}
\newtheorem{definition}[theorem]{Definition}

\newtheorem{example}[theorem]{Example}
\newtheorem{remark}[theorem]{Remark}

\newcommand{\parrafo}[1]{{\noindent \emph{#1}\;\;}}

\newcommand{\EE}{\mathbb E}

\newcommand{\NN}{\mathbb N}
\newcommand{\PP}{\mathbb P}
\newcommand{\R}{\mathbb R}
\newcommand{\RR}{\mathbb R}

\newcommand{\ZZ}{\mathbb Z}

\newcommand{\cA}{\mathcal A}

\newcommand{\cE}{\mathcal E}

\newcommand{\cN}{\mathcal N}

\newcommand{\cX}{\mathcal X}

\newcommand{\cW}{\mathcal W}

\newcommand{\sG}{\mathsf{G}}
\newcommand{\sM}{\mathsf{M}}
\newcommand{\sN}{\mathsf{N}}

\newcommand{\sU}{{\mathsf{U}}}
\newcommand{\sa}{{\mathsf{a}}}

\newcommand{\se}{{\mathsf{e}}}
\newcommand{\so}{{\mathsf{o}}}
\newcommand{\sg}{\mathsf{g}}

\newcommand{\su}{{\mathsf{u}}}

\newcommand{\hcW}{\widehat{\cW}}

\newcommand{\tu}{\tilde u}

\newcommand{\cXz}{{\check{\mathcal X}}}

\newcommand{\ckl}{{\check \ell}}

\newcommand{\cku}{{\check u}}

\newcommand{\ua}{\underline a}
\newcommand{\ub}{\underline b}

\newcommand{\uell}{\underline \ell}

\newcommand{\uu}{\underline u}

\newcommand{\one}{{\bf 1}\hskip-.5mm} 

\newcommand{\vep}{\varepsilon}
\newcommand{\eps}{\varepsilon}
\newcommand{\tsum}{\textstyle{\sum}}

\definecolor{lightgray}{rgb}{0.83, 0.83, 0.83}
\definecolor{lavendermist}{rgb}{0.9, 0.9, 0.98}
\definecolor{backgreen}{RGB}{82, 167, 106}
\definecolor{verdon}{RGB}{0, 100, 0}

\def\sqr{\vcenter{
		\hrule height.1mm
		\hbox{\vrule width.1mm height2.2mm\kern2.18mm\vrule width.1mm}
		\hrule height.1mm}} 
\def\square{\ifmmode\sqr\else{$\sqr$\vskip 3mm}\fi} 

\definecolor{cmm}{rgb}{0, .6, 0.4}
\definecolor{cmm'}{rgb}{.6, 0, .4}

\DeclareMathOperator{\Reflect}{Rf}

\allowdisplaybreaks
\title{\textsc{Slot decomposition of\\ continuous Box-Ball Systems}}
\date{\today}

\author{Inés Armendáriz, Pablo Blanc, \\
 Pablo A. Ferrari, Davide Gabrielli}

\begin{document}

\maketitle
	
\begin{abstract}
 \normalsize
We study a piecewise constant function $\eta:\RR\to\{-1,1\}$ with a finite number of discontinuities in any interval. We assume that the associated walk $\xi:\RR\to\RR$ satisfying $\xi'(x)=\eta(x)$, pinned by $\xi(0)=0$, has finite length excursions over past minima. This is the continuous generalization of an initial ball configuration in the discrete Box Ball System introduced by Takahashi and Satsuma, where solitons of integer sizes $k\ge1$ are identified. 

We extend the slot decomposition developed by Ferrari, Nguyen, Rolla and Wang in the discrete setting to the continuous case. Each soliton of $\xi$ is represented by a point in two dimensional space, one coordinate for position and the other for the soliton height, mapping $\xi$ to a point configuration. 

We consider a distribution on walks given by a product measure on the decomposition of the path into excursions over past minima. Excursions are distributed as products of their solitons weights, which are determined by the soliton heights. We show that when the weight function is in $L^1$ the slot decomposition of~$\xi$ is a Poisson process. This extends to the continuous case an approach of Ferrari and Gabrielli. As an example, we compute the intensity measure of the Poisson process associated to the asymmetric telegraph process introduced by Kac. In a forthcoming paper we discuss the dynamic properties.

\bigskip
		

\smallskip


\end{abstract}

\tableofcontents
\section{Introduction}\label{s1}

The Box-Ball System (BBS) is a prominent cellular automaton introduced by Takahashi and Satsuma \cite{TS90}. The model exhibits interesting solitonic behavior while remaining simple enough to be solved and analyzed in detail. Interest in the BBS has been recently revitalized due to growing connections between statistical mechanics and integrable systems. As an ultradiscrete version of the Korteweg-de Vries (KdV) equation, its connections with other models of integrable systems warrant further detailed investigation.

In the discrete setting, Ferrari, Nguyen, Rolla and Wang \cite{FNRW} introduced a description of the trajectories as a bi-infinite matrix with rows indexed by soliton sizes and columns indexed by the points, or slots,  where these solitons are inserted, a construction they labelled the slot decomposition. This representation has enabled the derivation of several features of the model, including hydrodynamic behavior by Croydon and Sasada \cite{CroydonSasada20}, a characterization of invariant measures by Ferrari and Gabrielli \cite{FG19}, and the fluctuations of soliton positions by Olla, Sasada and Suda \cite{olla2025scalinglimitssolitonsboxball}.

In this paper we consider a continuous version of the BBS. A configuration is a function $(\eta(t))_{t\in\RR}$ taking values $1$ and $-1$, with a finite number of discontinuities in bounded sets. Each $\eta$ is associated to a continuous real valued walk $\xi=(\xi(t))_{t\in\RR}$ with derivatives $\xi'(t)=\eta(t)$ at continuity points of $\eta$. We provide an algorithmic description of the slot decomposition of $\xi$, which yields a point configuration on a half-plane. Generalizing the discrete construction in \cite{FG19}, we identify the law of the point process associated with a random $\xi$ as a non-homogeneous two-dimensional Poisson process.
In future works we will address the relation between the slot decomposition of $\xi$ with its decomposition in terms of real trees, see for instance the lecture notes of Evans~\cite{evans}. We will also describe the evolution of the slot decomposition under the continuous version of the BBS dynamics of $\eta$, which is equivalent to Pitman's transformation of $\xi$~\cite{PitmanBM-Bessel}, as observed by Croydon and Sasada in~\cite{zbMATH07500585}.

There are several equivalent characterizations of solitons, see Ferrari and Gabrielli \cite{FG18} for a discussion. We describe a simple algorithm that is particularly well suited to continuous walks. In this representation, every soliton is centered at a local maximum of $\xi$. Solitons are identified independently within each excursion over past minima of the walk. The procedure starts with the shortest solitons and then proceeds recursively to larger ones.

To each soliton we associate a pair $(u,\ell)$, where $\ell$ is the height of the soliton and $u$ encodes its relative position within the walk.
In this way, a zig-zag walk $\xi$ is identified with a point process $ \sN[\xi]\subset \RR\times \RR_+$, which we call the walk's slot decomposition;
see Fig.~\ref{fig:intro}.
This construction defines a bijection between an appropriate space of walks $\hcW$ and a corresponding space of point configurations $\cN$.

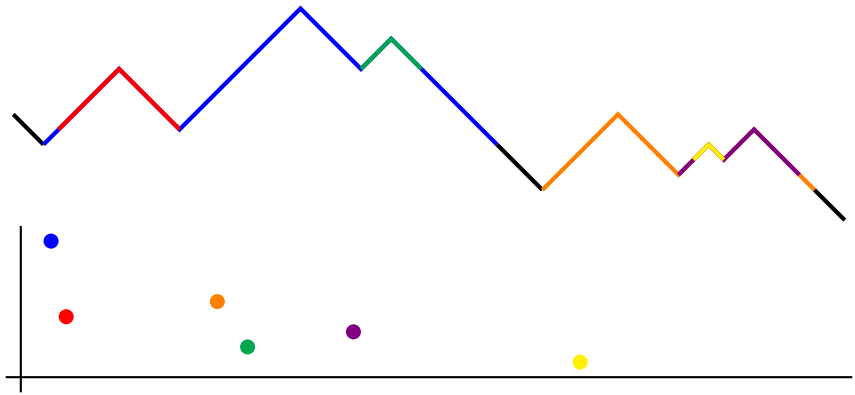
\begin{figure}
\centering
\begin{tikzpicture}[scale=0.2, thick]

    \draw[ultra thick] (0,0) -- (2,-2);
    
\begin{scope}[shift={(2,-2)}]
    \draw[blue, ultra thick] (0,0) -- (5,5) -- (9,1) -- (17,9) -- (21,5) -- (23,7) -- (30,0);
    \draw[red, ultra thick] (1,1) -- (5,5) -- (9,1) ;
    \draw[Green, ultra thick] (21,5) -- (23,7) -- (25,5);
\end{scope}

    \draw[ultra thick] (32,-2) -- (35,-5);
	\draw[orange,ultra thick] (35,-5) -- (40,0)-- (44,-4) -- (46,-2) -- (47,-3) -- (49,-1)-- (53,-5);  
	\draw[violet,ultra thick] (44,-4) -- (46,-2) -- (47,-3) -- (49,-1)-- (52,-4);   
	\draw[yellow,ultra thick] (45,-3) -- (46,-2) -- (47,-3);    
	\draw[ultra thick] (53,-5) -- (55,-7);

\end{tikzpicture}

\begin{tikzpicture}[scale=0.2, thick]

\draw (-1,0) -- (55,0);
\draw (0,-1) -- (0,10);

\begin{scope}[shift={(2,0)}]
	\node[dot, blue] at (0,9) {};	
	\node[dot, red] at (1,4) {};	
	\node[dot, Green] at (13,2) {};

	 \begin{scope}[shift={(3,0)}]
	 	\node[dot, orange] at (8,5) {};
	 	\node[dot, violet] at (17,3) {};
	 	\node[dot, yellow] at (32,1) {};
	 \end{scope}	   
\end{scope}	    

\end{tikzpicture}

\caption{On top, a walk $\xi$  with solitons identified by different colors. Below, the associated slot decomposition $\sN[\xi]$. }
\label{fig:intro}
\end{figure}

We define a measure over the space of walks.
For each excursion $\vep$ we define the measure $\mu_\alpha(\vep)$ proportional to $\prod_\gamma \alpha(\ell_\gamma)$, where the product runs over the solitons $\gamma$ of the excursion, $\ell_\gamma$ denotes the height of $\gamma$, and $\alpha$ is a prescribed parameter function.
We consider a random walk $\xi$ conditioned to have a record at the origin and concatenated iid excursions with law $\mu_\alpha$. The excursions are separated by intervals of records over past minima with lengths that are distributed as iid exponential random variables; the interval of records containing the origin is the sum of two of those variables.
We prove that the associated slot decomposition $\sN[\xi]$ is a Poisson point process on $\mathbb{R}\times\mathbb{R}_+$ with a non-homogeneous intensity measure $dx\, q(y)\,dy$, where the intensity function $q$ is explicitly related to the parameter function $\alpha$.

An important example is the asymmetric telegraph process, whose derivative $\xi'$ is a stationary continuous pure jump  Markov process on $\{-1,1\}$ with rates $\lambda_-$ to jump from $-1$ to $1$ and $\lambda_+$ for the reverse jump, and $\lambda_+>\lambda_-$. We show that in this case the excursions of the trajectory have law $\mu_{\alpha}$ with $\alpha(\ell)= \lambda_+\lambda_-e^{(\lambda_++\lambda_-)\ell}$, and the intervals of records connecting the excursions are iid exponential $\lambda_-$ random variables. We compute explicitly the Poisson intensity $q$ as a function of $\lambda_\pm$. The telegraph process was introduced by  Kac \cite{zbMATH03491985}, see the book of Ratanov and Kolesnik \cite{zbMATH07582508} for the asymmetric version.

 Our approach is based on the slot decomposition introduced in \cite{FNRW} for the discrete case, and the articles \cite{FG19,FG18} that establish a bijection between measures based on soliton weights and product measures on bi-infinite matrices $\chi:\ZZ\times \NN\to\NN\cup\{0\}$.

We mention some references in the discrete BBS literature that are relevant for the slot decomposition. Hambly, Martin and O'Connell \cite{HamblyMartinOConnell01} proved that stationary Markov measures on $\{-1,1\}^\ZZ$ are invariant for Pitman's transformation; see also
 \cite{cs-markov} and \cite{FNRW}. In the continuous setup, Markov corresponds to the, possibly asymmetric, telegraph process. 
Mucciconi, Sasada, Sasamoto, Suda \cite{mss2023} have recently proved that the slot decomposition in \cite{FNRW} is equivalent to the Kerov-Kirillov-Reschetikhin (KKR) bijection studied by  Kuniba, Okado, Sakamoto,Takagi, Yamada \cite{kosty2006}, see also Kirillov-Sakamoto \cite{KS2009} and Mada, Idzumi, Tokihiro \cite{MIT}. 

Related dynamics and its scaling limits include Hard rods system by Boldrighini, Dobrushin, Suhov \cite{bds,BS}, Ferrari, Franceschini, Grevino, Spohn \cite{ffgs}, Ferrari, Olla \cite{fo2022}; Toda Lattice by Spohn \cite{zbMATH07852571}, Cao, Bulchadani, Spohn \cite{cbs}, Croydon, Sasada and Tsujimoto \cite{zbMATH07653279,zbMATH07620722}, Aggarwal, Nicoletti \cite{arXiv:2604.14346}.   See also the surveys Bulchandani-Cao-Moore \cite{bulchandani-cao-moore}, Doyon \cite{Doyon_2019}, Spohn \cite{Spohn_2019,spohn2023}.

The article is organized as follows.
In Section~\ref{s1} we perform the identification of solitons and the slot decomposition of a continuous ball configuration, and describe the sets of walks and point configurations where the decomposition is a bijection. In Section~\ref{s40} we consider random excursions, introduce the measure $\mu_\alpha$, a Poisson process with intensity $dxq(y)dy$, and establish conditions on $\alpha$ and $q$ under which these measures are equivalent. In Section~\ref{ss32} we extend both measures to walk configurations and explore the relation between both representations. In Appendix~\ref{appendix} we discuss soliton ties.

\section{Soliton decomposition and slot representation}
\label{s1}

The goal of this section is to construct a bijection between walks and point configurations.
After introducing the necessary definitions in Subsection~\ref{ss11}, we 
identify the solitons in a walk, one per local maximum.
We then construct the point configuration by associating  a $2$-dimensional point to each soliton, whose coordinates are determined by the insertion position of the soliton in the walk and its height.
Finally we show how to reconstruct the walk from the point configuration. We do this construction for a single, isolated excursion in a first step, Subsection~\ref{sde2}, and then extend it to the full trajectory. 

In several stages of the construction we need to identify the points realizing the absolute extrema of a trajectory over a given interval. Where necessary, we define explicit tie-breaking rules. However, in the absence of such instructions, any valid selection may be used. Detailed analyses of specific cases are provided in the Appendix, which serves as a reference for readers seeking further clarification on these subtleties.

\subsection{Ball configurations, walks and carrier load}
\label{ss11}

\parrafo{Continuous ball configurations} Let
\begin{align*}
  \cXz\vcentcolon =
  \{\eta:\RR\to\{-1, 1\}: \text{càdlàg with finitely many jumps in bounded intervals}\}.
\end{align*}
A function $\eta\in\cXz$ is characterized by the sets of discontinuity points $A[\eta]$ and $B[\eta]$, given by
\begin{align}
 \label{ab1}
 A[\eta] \vcentcolon = \{x\in\RR: \eta(x)-\eta(x-)=2\}, \quad B[\eta] \vcentcolon= \{x\in\RR: \eta(x)-\eta(x-)=-2\},
\end{align}
where $\eta(x-)\vcentcolon =\lim_{y\uparrow x} \eta(y)$. 
We interpret the interval between an upward jump and the nearest downward jump to its right as a continuous set of balls, in analogy with the discrete BBS.

\parrafo{Walk representation} Given a ball configuration $\eta$ denote by  $\xi=\xi[\eta]$ the associated zig-zag walk, a continuous function pinned at the origin, $\xi(0)\vcentcolon =0$, with derivatives
\begin{align}
 \xi'(x) \vcentcolon = \eta(x), \quad x\notin A[\eta]\cup B[\eta]. \label{z23}
\end{align}
The elements of the sets $A[\eta]$ and $B[\eta]$ are the local minima and local maxima of $\xi$, respectively. We also refer to them as $A[\xi]$ and $B[\xi]$. 
Let $\cW$ be the set of walks given by
\begin{align}
  \cW \vcentcolon=\left\{\xi[\eta]: \liminf_{x\to\mp\infty} \xi[\eta](x) = \pm \infty, \,\eta\in\cXz\right\}, \label{cZ}
\end{align}
and let $\cX$ be the set of ball configurations associated to walks in $\cW$, that is
\begin{align}
\cX \vcentcolon = \bigl\{\eta[\xi]: \xi\in \cW\bigr\}. \label{cX}
\end{align}

\parrafo{Past minimum} For $\xi\in\cW$ we define the past minimum function $m=m[\xi]:\RR\to \RR$ by 
\begin{align}
 \label{lm9}
 m(x) \vcentcolon= \min\bigl\{\xi(z):z\le x\bigr\}.
\end{align}

\parrafo{Records} We say that $r\in\RR$ is a record for $\xi$ if $r$ is the first time the path reaches level $\xi(r)$, that is, $m(r)=\xi(r)$ and $\xi(x)>\xi(r)$ for every $x<r$. The $y$-record is defined by
\begin{align}
 \label{rec1}
 r(y) &\vcentcolon= \inf \xi^{-1}(y).
\end{align}
Notice that $\xi^{-1}(y)$ is non-empty and $r(y)>-\infty$ for $\xi\in\cW$.
Denote the set of records by
\begin{align}
  \label{R28}
 R[\xi]&\vcentcolon= \{r(y):y\in\RR\}.
\end{align}

Let $\hcW$ be the set of walks in $\cW$ having record $0$ at the origin,
\begin{align}
  \label{hcW}
  \hcW \vcentcolon = \{\xi\in \cW: \xi(0)=0<\xi(x),\,x<0 \}.
\end{align}

\parrafo{Excursions.} When $\xi\in\cW$, the mapping $y\mapsto m^{-1}(y)$ associates to each ordinate $y$ a set containing either a single point or the interval $[r(y), r(y-)]$. The restriction of the walk
\begin{align}
  \label{ex29}
 \big (\xi(x):x\in (r(y), r(y-)]\bigr), 
\end{align}
is called the excursion at level $y$. If $m^{-1}(y)=\{r(y)\}$, then $r(y)=r(y-)$, and the excursion at level $y$ is the empty set. By construction, $\xi(r(y))=\xi(r(y-))$, and if the excursion is not empty, then $\xi(x)\ge \xi(r(y)),\,x\in (r(y), r(y-)]$. We use the word excursion to refer to a non empty excursion.

\parrafo{Carrier load process.} The carrier load process $\zeta[\xi]$ is defined by
\begin{align}
  \label{clp1}
  \zeta(x) \vcentcolon = \xi(x)-m(x), \qquad x\in\RR.
\end{align}
Conversely, 
the increments of $\xi$ are determined by $\zeta$,
\begin{align}
  \label{4}
  \xi(x)-\xi(x') = \zeta(x)-\zeta(x')- |\{z\in [x,x']:\zeta(z)=0\}|,\qquad x'<x.
\end{align}
The trajectories $\xi$ and $\zeta$ have the same excursions above the past minimum, and the flat regions of $\zeta$ correspond to records of $\xi$.

\subsection{Soliton identification}
\label{siw3}

We identify the solitons within each excursion iteratively: we first identify the shortest soliton, remove it, we then identify the shortest among the remaining solitons, and so on.

Fix an excursion $\eps=\big (\xi(x):x\in (r(y), r(y-)]\bigr)$ and let
\begin{align}
  \label{9}
(\so,\se]=(r(y), r(y-)]. 
\end{align}

The end points of the interval are called the origin and the end of the excursion.
The set of local minima is denoted  $A[\eps]=\{a_1^1, \dots, a_n^1\}$, while $B[\eps]=\{b_1^1, \dots, b_n^1\}$ is the set of local maxima.
They satisfy $o=a_1^1<b_1^1<a_2^1<\dots<a_n^1<b_n^1<\se$. They determine $2n$ open intervals that do not contain extrema. We will refer to the first $2n-1$ of these intervals as runs. They are given by
\begin{equation*}
(a_i^1, b_i^1),\,i=1, \dots, n, \qquad\text{and}\qquad (b_i^1,a_{i+1}^1), \, i=1, \dots, n-1.
\end{equation*}

We order the runs by increasing length. When there are two or more runs of the same length, we sort them by their left endpoint; for simplicity, we will refer to the first of the runs, in this order, as the shortest run.
Note that since $\eps(\se)\leq \eps(a_n)$ we have that $(a_n^1, b_n^1)$ is necessarily shorter than $(b_n^1,\se)$. This last interval is technically not a run, in fact, it will be part of a longer interval that includes the interval of records separating $\eps$ from the next excursion in the trajectory. 
We conclude that the shortest run is $(a,b)$ or $(b,a)$ for some $a\in A[\eps]$ and $b\in B[\eps]$.

Denote the shortest run height by
\begin{align*}
  \ell\vcentcolon = |b-a|.
\end{align*}
Define the (shortest) soliton $\gamma_1$ by
\begin{gather}
 \label{sol1}
 \gamma_1 \vcentcolon = (b-\ell, b+\ell], 
\end{gather}
the interval with center $b$ and radius $\ell$. Denote
\begin{align}
  a_1\vcentcolon =a, \quad b_1\vcentcolon =b,\quad \ell_1\vcentcolon =\ell,\label{sol3}
\end{align}
the minimum,  maximum, and height of the shortest soliton $\gamma_1$.
The length of the soliton is $2\ell_1$.
See Fig.~\ref{solid1}.

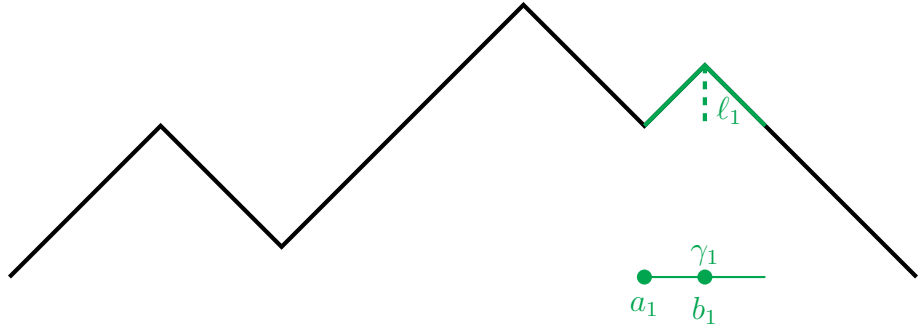
\begin{figure}[th]
\centering
\begin{tikzpicture}[scale=0.4, thick]
    \draw[ultra thick] (0,0) -- (5,5) -- (9,1) -- (17,9) -- (21,5) -- (23,7) -- (30,0);
    
    \draw[Green, ultra thick] (21,5) -- (23,7) -- (25,5);

    \node[dot, Green,label={[Green]below:{$b_1$}}] at (23,0) {};
    \node[dot, Green,label={[Green]below:{$a_1$}}] at (21,0) {};
    
    \draw[dashed, Green, ultra thick] (23,7) -- (23,5);
    \node[Green,right] at (23,5.5) {$\ell_1$};
    
    \draw [Green] (21,0) -- (25,0);
    \node[Green,above] at (23,0) {$\gamma_1$};

\end{tikzpicture}
\caption{First step of the soliton identification process. Here $\gamma_1=(a_1, a_1+2\ell_1]$.}\label{solid1}
\end{figure}

If there are no more maxima, $B[\eps]=\{b_1\}$,  then $\eps$ has only one soliton, and we define $\Gamma[\eps]=\{\gamma_1\}$.
Otherwise, we cut out previously identified solitons and iterate.
In what follows we drop the dependence of $\eps$ when clear from context.

After $k$ steps we have $\Gamma_k=\{\gamma_1,\gamma_2, \dots,\gamma_k\}$.
We will work with the excursion obtained after removing the previously identified solitons.
To avoid moving every point, we introduce a family of metric spaces indexed by the excursion $\eps$ and the number of previously identified solitons.
See Fig.~\ref{firststep}.

Define the metric space $(R_k, d_k)$, where $R_k$ is the set of points not occupied by the previous identify solitons, and $d_k$ is the distance restricted to $R_k$, 
\begin{align}
 \label{rel1}
 R_k&\vcentcolon =(\so,\se]\setminus \cup_{\gamma\in \Gamma_k}\gamma, \qquad d_k(x, z) \vcentcolon = \bigl|[x, z ]\cap R_k\bigr|, 
\end{align}
for $x<z$ in $R_k$.
Here $|\cdot|$ denotes the Lebesgue measure on $\RR$.
By symmetry, $d_k(z, x)\vcentcolon =d_k(x, z)$.

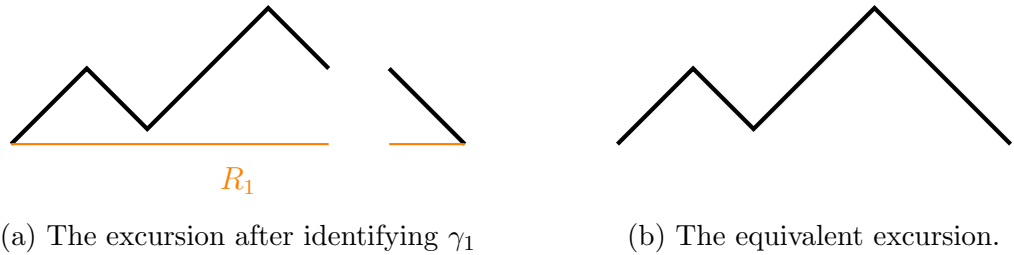
\begin{figure}[th]
    \centering
    \begin{subfigure}[b]{0.48\textwidth}
        \centering
        \begin{tikzpicture}[scale=0.2, thick]
    \draw[ultra thick] (25,5) -- (30,0);
    \draw[ultra thick] (0,0) -- (5,5) -- (9,1) -- (17,9)-- (21,5);

    \node[label={[orange]below:{$R_1$}}] at (15,0) {};

    \draw [orange] (0,0) -- (21,0);
    \draw [orange] (25,0) -- (30,0);
\end{tikzpicture}
        \caption{The excursion after identifying $\gamma_1$}
    \end{subfigure}
    \begin{subfigure}[b]{0.48\textwidth}
        \centering
        \begin{tikzpicture}[scale=0.2, thick]
    \draw[ultra thick] (0,0) -- (5,5) -- (9,1) -- (17,9)-- (26,0);
        \node[label={[white]below:{$R_1$}}] at (15,0) {};
\end{tikzpicture}
        \caption{The equivalent excursion.}
    \end{subfigure}
    \caption{The excursion after removing one soliton}
    \label{firststep}
\end{figure}

We now identify the solitons inductively and prove the following claim: $\eps$ restricted to $(R_k, d_k)$ is a continuous zig-zag excursion with minima and maxima given by $A_{k+1}=A\setminus \{a_1, \dots, a_k\}=\{a^{k+1}_1,\dots,a_{n-k}^{k+1}\}$ and
$B_{k+1}=B\setminus \{b_1, \dots, b_k\}=\{b^{k+1}_1,\dots,b_{n-k}^{k+1}\}$.
And they satisfy $o=a_1^{k+1}<b_1^{k+1}<a_2^{k+1}<\dots<a_{n-k}^{k+1}<b_{n-k}^{k+1}<e$.

We identify the shortest run with respect to $d_k$, say $(a,b)$ (the case $(b,a)$ is analogous).
Let $\ell= d_k(a, b)$.
Observe that since $\eps$ is a zig-zag we have $\ell= \eps(b)-\eps(a)$.

Define
\begin{align}
 \gamma_{k+1}&\vcentcolon = \bigl(b-x_{k+1},b+z_{k+1}\bigr]\;\cap\; R_k, \label{sol6}
 \\
 z_{k+1} &\vcentcolon =\text{solution $z>0$ of } d_{k}(b, b+z)=\ell,\notag
 \\
 x_{k+1} &\vcentcolon =\text{solution $x>0$ of } d_{k}(b-x,b)=\ell.\notag
\end{align}
In other words, $\gamma_{k+1}$ is the interval contained in $R_k$ with the $d_k$ metric, centered at $b$ with radius $\ell$.
Observe that $b-x_{k+1} =a$.  

Since $(a,b)$ is the shortest run there is no minimum in $(b,b+z_{k})$.
Recalling that $\eps$ restricted to $(R_k, d_k)$ is a continuous zig-zag we get $\eps(a)=\eps(b+z_{k})$.
Therefore the resulting path after removing $\gamma_{k+1}$ is a continuous zig-zag.

Observe that $a$ and $b$ are the only minimum and maximum in $\gamma_{k+1}$.
Also, given that $\tilde b < a < b < \tilde a$ are consecutive min/max then we have $b+z_{k}\leq\tilde a$ and $\eps$ is decreasing in $(\tilde b,\tilde a)\cap (R_k\setminus \gamma_{k+1})$.
We conclude that the min/max of $\eps$ restricted to $R_k\setminus \gamma_{k+1}$ are exactly $A_{k+1}\setminus \{a\}$ and $B_{k+1}\setminus \{b\}$.
We have proved our claim and conclude the soliton identification.
See Fig.~\ref{solifull}.

Note that since $(a,b)$ is the shortest, every other run is equal in length or longer. Also, 
we have $\eps(\tilde b)>\eps(b)>\eps(a)\ge\eps(\tilde a)$, and it follows that the run that appears after removing $\gamma$, that is $(\tilde b,\tilde a)$, is longer than 
$(a,b)$.
We conclude that 
\begin{equation}
\ell_1\leq \ell_2\leq \cdots\leq \ell_n.
\label{orderedell}
\end{equation}

In the last step of the identification, when there is only one remaining maximum, we get that $\gamma$ is equal to the whole $R_k$, so we conclude that
\begin{align}
  \label{1}
 (\so,\se]=\bigcup_{b\in B[\eps]}\gamma_b. 
\end{align}
By \eqref{sol6}, the union is disjoint. Since over the soliton with height $\ell_i$ the excursion climbs up and then down on intervals that add to a total length $2 \ell_i$, it easily follows that
\begin{align}
\label{eq:4}
 \sum_{i=1}^n (b_i-a_i)= \sum_{i=1}^{n-1}(a_{i+1}-b_i) + (\se-b_n)= \sum_{i=1}^n \ell_i.
\end{align}

\begin{figure}[th]
\centering
\begin{tikzpicture}[scale=0.4, thick]
    \draw[ultra thick] (0,0) -- (5,5) -- (9,1) -- (17,9) -- (21,5) -- (23,7) -- (30,0);
    
    \draw[blue, ultra thick] (0,0) -- (5,5) -- (9,1) -- (17,9) -- (21,5) -- (23,7) -- (30,0);
    \node[dot, blue,label={[blue]below:{$b_3$}}] at (17,0) {};
    \node[dot, blue,label={[blue]below:{$a_3$}}] at (0,0) {};
    \draw[dashed, blue, ultra thick] (17,0) -- (17,9);
    \node[blue,right] at (17,5.5) {$\ell_3$};
    \draw [blue] (0,0) -- (30,0);
    \node[blue,above] at (15,0) {$\gamma_3$};
    
    \draw[red, ultra thick] (1,1) -- (5,5) -- (9,1) ;
    \node[dot, red,label={[red]below:{$b_2$}}] at (5,0) {};
    \node[dot, red,label={[red]below:{$a_2$}}] at (9,0) {};
    \draw[dashed, red, ultra thick] (5,1) -- (5,5);
    \node[red,right] at (5,2.5) {$\ell_2$};
    \draw [red] (1,0) -- (9,0);
    \node[red,above] at (3,0) {$\gamma_2$};
    
    \draw[Green, ultra thick] (21,5) -- (23,7) -- (25,5);
    \node[dot, Green,label={[Green]below:{$b_1$}}] at (23,0) {};
    \node[dot, Green,label={[Green]below:{$a_1$}}] at (21,0) {};
    \draw[dashed, Green, ultra thick] (23,7) -- (23,5);
    \node[Green,right] at (23,5.5) {$\ell_1$};
    \draw [Green] (21,0) -- (25,0);
    \node[Green,above] at (24,0) {$\gamma_1$};

\end{tikzpicture}
\caption{The soliton identification.}\label{solifull}
\end{figure}
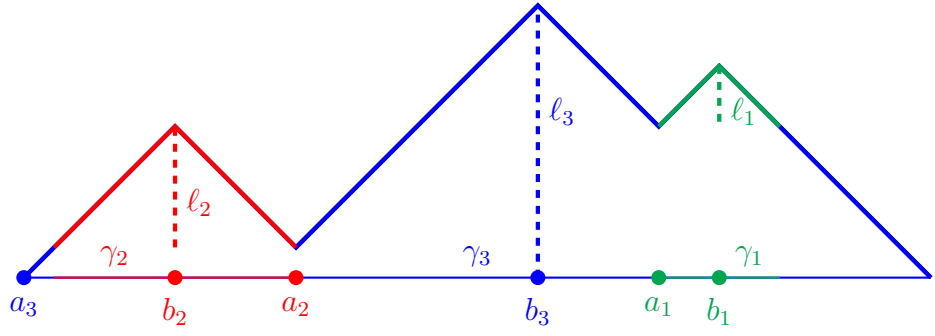

Given $b\in B[\xi]$ we have that $b\in B[\eps]$ for some excursion $\vep$.  
We denote by $\ell_b$, $a_b$ and $\gamma_b$ the size, minimum and soliton associated to $b$, obtained by the algorithm applied to the excursion $\vep$. 
We also denote the origin and end of the soliton by
\begin{align}\label{obeb}
o_b=\inf \gamma_b \quad \text{and} \quad e_b=\max \gamma_b.
\end{align}

\subsection{Slot decomposition}
\label{subsecsol}

We introduce a family of metric spaces indexed by the walk $\xi\in\cW$ and soliton heights. Given a height $\ell$, define the metric space $(\RR_\ell, d_\ell)$, where $\RR_\ell=\RR_\ell[\xi]$ is the set of points not occupied by solitons $\gamma_b$ with $\ell_b\leq\ell$, and $d_\ell=d_\ell[\xi]$ is the distance restricted to $\RR_\ell$, 
\begin{align}
 \label{rel1}
 \RR_\ell&\vcentcolon =\RR\setminus \cup_{b\in B, \ell_b\le \ell}\,\gamma_b, \qquad d_\ell(x, z) \vcentcolon = \bigl|[x, z ]\cap \RR_\ell\bigr|, 
\end{align}
for $x<z$ in $\RR_\ell$; by symmetry, $d_\ell(z, x)\vcentcolon =d_\ell(x, z)$.
The space $\RR_\ell$ includes every record and for each excursion the points in $R_k$ where $k$ is the number of solitons of height less that $\ell$ within the excursion.

Define $\xi_\ell:\RR_\ell\to\RR$, the walk restricted to $\RR_\ell$, by
 \begin{gather}
 \label{xil1}
  \xi_\ell(x) =\xi(x), \; x\in \RR_\ell.
\end{gather}
Notice that $\xi_\ell$ is a continuous function with respect to the distance $d_\ell$.

For $b\in B[\xi_\ell]$, define the interval
\begin{align}
 I_\ell(b)&\vcentcolon = \bigl(b-\tilde x,b+\tilde z\bigr]\label{defI}
 \\
 \tilde z &\vcentcolon =\text{solution $z>0$ of } d_\ell(b, b+z)=\ell\notag
 \\
 \tilde x &\vcentcolon =\text{solution $x>0$ of } d_\ell(b-x,b)=\ell.\notag
\end{align}

Define the $\ell$ slot space $S_\ell=S_\ell[\xi]$ by
\begin{align}
  \label{2}
  S_\ell\vcentcolon = \RR_\ell \setminus \cup_{b\in B[\xi_\ell]} I_\ell(b).
\end{align}
See Fig.~\ref{alter}.

\begin{figure}[th]
\centering
\begin{tikzpicture}[scale=0.2, thick]

    \draw[ultra thick] (0,0) -- (2,-2);
    
\begin{scope}[shift={(2,-2)}]
    \draw[blue, ultra thick] (0,0) -- (5,5) -- (9,1) -- (17,9) -- (21,5) -- (23,7) -- (30,0);
    \draw[red, ultra thick] (1,1) -- (5,5) -- (9,1) ;
    \draw[Green, ultra thick] (21,5) -- (23,7) -- (25,5);
\end{scope}

    \draw[ultra thick] (32,-2) -- (35,-5);
	\draw[orange,ultra thick] (35,-5) -- (40,0)-- (44,-4) -- (46,-2) -- (47,-3) -- (49,-1)-- (53,-5);  
	\draw[violet,ultra thick] (44,-4) -- (46,-2) -- (47,-3) -- (49,-1)-- (52,-4);   
	\draw[yellow,ultra thick] (45,-3) -- (46,-2) -- (47,-3);    
	\draw[ultra thick] (53,-5) -- (55,-7);

    \begin{scope}[shift={(0,-20)}]
    
    \draw[fill=Green!30,draw=none] (25,0) -- (27,2) -- (27,13)-- (23,13)-- (23,2) -- cycle;
    \draw[fill=red!30,draw=none] (7,0) -- (11,4) -- (11,13)-- (3,13)-- (3,4) -- cycle;
    
    \draw[fill=blue!30,draw=none] (19,0) -- (23,4) -- (23,13)-- (11,13)-- (11,8) -- cycle;
    \draw[fill=blue!30,draw=none] (3,8) -- (2,9) -- (2,13) -- (3,13) -- cycle;
    \draw[fill=blue!30,draw=none] (27,4) -- (32,9) -- (32,13)-- (27,13)--  cycle;

	\draw[fill=yellow!30,draw=none] (46,0) -- (47,1) -- (47,13)-- (45,13)-- (45,1) -- cycle;

	\draw[fill=purple!30,draw=none] (49,0) -- (52,3) -- (52,13)-- (47,13)-- (47,2) -- cycle;   
	\draw[fill=purple!30,draw=none] (45,2) -- (44,3)-- (44,13) -- (45,13) -- cycle;
	
	\draw[fill=orange!30,draw=none] (40,0) -- (44,4) -- (44,13)-- (35,13)-- (35,5) -- cycle;  
	
	\draw[fill=orange!30,draw=none] (52,4) -- (53,5) --(53,13) -- (52,13)  -- cycle;

    \draw (-1,0) -- (55,0);
    \draw (0,-1) -- (0,15);
    \node[left] at (0,7) {$\ell$};
\end{scope}

\end{tikzpicture}
\caption{Draw a horizontal line at level $\ell$. Then $S_\ell$ is given by the portions of the line that intersect the region in white.
It is obtained by removing the solitons with height smaller that $\ell$ and the $I_\ell$ regions associated to solitons of height greater than $\ell$.}
\label{alter}
\end{figure}
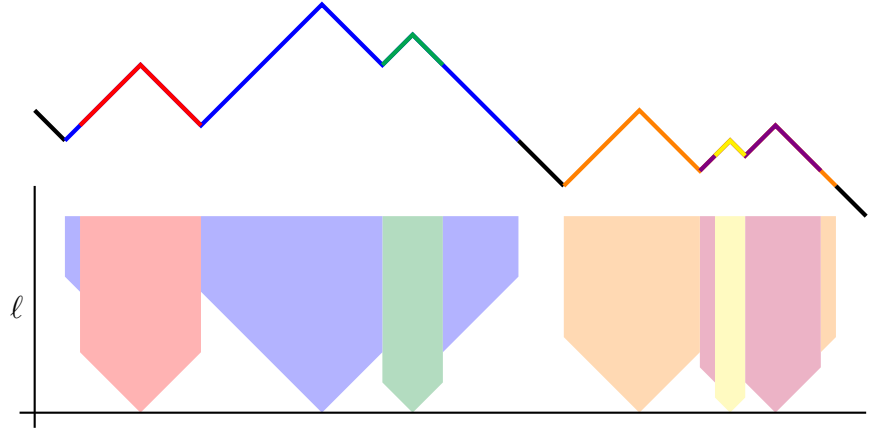

Consider the distance $h_\ell$ on $S_\ell$ given by
\begin{align}
 \label{dhs1}
 h_\ell(x, z)&\vcentcolon = \bigl| [x, z]\cap S_\ell\bigr|, \qquad x<z, \;x, z\in S_\ell.
\end{align}
and by symmetry, $h_\ell(z, x)\vcentcolon =h_\ell(x, z)$.
Recall the definition of the set $\hcW$ in \eqref{hcW}. The slot decomposition of $\xi\in\hcW$ is the discrete set $\sN=\sN[\xi]\subset\RR\times \RR_+$ given by
\begin{align}
 \sN &\vcentcolon =\sN_+\cup\sN_-, \label{nxi1}\\
  \sN_+&\vcentcolon = \bigl\{\bigl( h_{\ell_b}(0, o_b)\, , \, \ell_b\bigr) :b\in B, \, b>0 \bigr\}
         \label{nxi+}\\
  \sN_-&\vcentcolon = \bigl\{\bigl(-h_{\ell_b}(o_b,0 )\, , \, \ell_b\bigr) :b\in B, \, b<0 \bigr\},
         \label{nxi-}
\end{align}
$o_b$ as in \eqref{obeb}.
See Fig.\/~\ref{inverse1}.

We use here the set notation although $\sN$ is in fact a multiset.
Points may appear with multiplicity; see the appendix for details.

{ \centering
 \includegraphics
 [width=.95\textwidth]
 {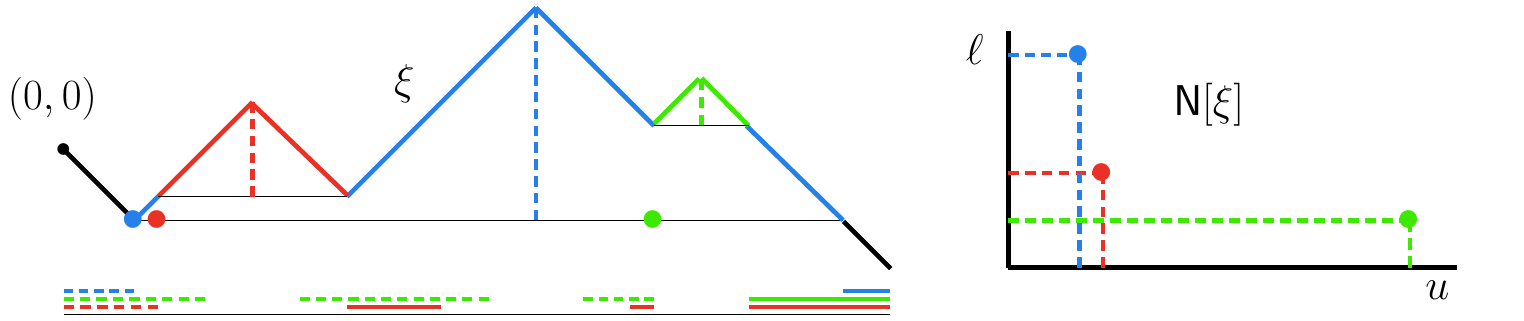}
 \captionof{figure}{Left: an excursion  $\xi$ with 3 solitons. Each vertical line segment represents the height $\ell_b$ of the soliton $\gamma_b$. The 3 bearers $o_b$ are represented by dots.  Dashed horizontal segments depict the slot space in the interval $[0, o_b]$ associated to each soliton height, and solid segments represent additional slot space. Right: Each maximum $b$ of $\xi$ is mapped to a point $(u_b,\ell_b)$ in $\sN[\xi]$, depicted by a dot. The horizontal coordinate $u_b$  is the $h_{\ell_b}$ distance between the origin and the soliton bearer $o_b$. }\label{inverse1}
}

\begin{lemma}[Soliton exclusion rule]
 \label{exc1}
Let $\xi\in\cW$, $\ell>0$, and $b,b'\in B[\xi]$ with $\ell<\ell_b\le \ell_{b'}$. Then 
\begin{equation}
I_\ell(b)\cap I_\ell(b')=\emptyset
\label{5}
\end{equation}
\end{lemma}
\begin{proof}
There is a minimum in $\RR_{\ell}$, say $a$, between the maxima $b$ and $b'$. By \eqref{orderedell}, we have $d_{\ell}(b, a)\ge \ell$, $d_{\ell}(b', a)\ge \ell$, implying $d_{\ell}(b,b')> 2\ell$ and therefore \eqref{5}.
\end{proof}

For later use, we observe that the amount of $\ell$ slot space of an excursion $\eps$ with domain $(\so,\se]$ is given by
\begin{equation}
h_\ell(\so,\se)=|(\so,\se]\cap S_\ell|=2\sum_{b\in B[\eps]}(\ell_b-\ell)^+.
\label{eqhelloe}
\end{equation}
Indeed, each soliton with height larger than $\ell$ contributes $2(\ell_b-\ell)^+$
to the $\ell$ slot space.
We have used that the excursion is the disjoint union of its solitons, as given by \eqref{1}, and the soliton exclusion rule \eqref{5}.

In the remaining of this subsection we define a space $\cN$ containing the image of the map $\xi\mapsto \sN[\xi]$. We show later, in Section \ref{ipc1}, that $\xi\mapsto \sN[\xi]$ is a bijection between $\hcW$ and $\cN$. 

Given $y>0$ we consider the path from 0 to $x=r((-y)-)$.
We have
\begin{equation}
h_\ell(0,x)=|(0,x]\cap S_\ell|=y+2\sum_{b\in B[\xi]\cap (0,x]}(\ell_b-\ell)^+.
\label{formulahell}
\end{equation}
That is, the $\ell$ slot space is given by the contributions of the records, that is $y$, plus the contributions of all the excursions.

With $x$ and $y$ as above, let
\begin{equation}
\mathcal T(y)= \{(u,\ell)\in \sN_+: u\leq h_\ell(0,x)\}.
\label{defRy}
\end{equation}
Given $b\in B[\xi]\cap (0,x]$ we have $o_b\leq x$ so
$h_{\ell_b}(0,o_b)\leq h_{\ell_b}(0,x)$.
And in the same way for $b\in B[\xi]\cap (x,+\infty)$ we have  
$h_{\ell_b}(0,o_b)> h_{\ell_b}(0,x)$.
Therefore $\mathcal T(y)$ is given by the points corresponding to maxima in $B[\xi]\cap (0,x]$, that is
\[
\mathcal T(y)= \bigl\{\bigl( h_{\ell_b}(0, o_b)\, , \, \ell_b\bigr) :b\in B[\xi]\cap (0,x] \bigr\}
\]
which is finite since $B[\xi]\cap (0,x]$ is finite.
We can write
\begin{equation}
\mathcal T(y)= \{(u,\ell)\in \sN_+: u\leq y+2\sum_{(\tu,\tilde \ell) \in \mathcal T(y)}(\tilde \ell-\ell)^+\}.
\label{eqRy}
\end{equation}

Given $\sN\subset\RR\times \RR_+$, let 
\begin{align}\label{n+n-}
\sN_+=\{(u,\ell)\in \sN: u\geq 0\}
\quad\text{and}\quad
\sN_-=\{(u,\ell)\in \sN: u< 0\}.
\end{align}
These definitions are consistent with the ones provided in \eqref{nxi+} and \eqref{nxi-}. Indeed, 
 if $\sN_-$ is given as in \eqref{nxi-}, then every pair $(u,\ell) \in \sN_-$ satisfies $u<0$ due to the fact that $0$ is a record.

We say that $\sN_+$ is $\mathcal T$-finite if for every $y>0$ there exists a finite set $\mathcal T(y)$ that verifies \eqref{eqRy}.
Analogously we say that $\sN_-$ is $\mathcal T$-finite if for every $y>0$ there exists a finite set $\mathcal T(y)$ that verifies
\begin{equation*}
\mathcal T(y)= \{(u,\ell)\in \sN_-: u\geq -y-2\sum_{(\tu,\tilde \ell) \in \mathcal T(y)}(\tilde \ell-\ell)^+\}.
\end{equation*}
Finally we say $\sN$ is $\mathcal T$-finite if both 
$\sN_+$ and $\sN_-$ defined in \eqref{n+n-} are $\mathcal T$-finite. Define
\begin{align}
  \label{7}
  \cN  \vcentcolon = \text{set of point configurations that are
$\mathcal T$-finite}.
\end{align}
Since we have shown that $\mathcal T(y)$ given by \eqref{defRy} verifies \eqref{eqRy} we have proved the following proposition.

\begin{proposition}
For every $\xi\in\hcW$ we have
\[
\sN[\xi]\in \cN.
\]
\end{proposition}

\subsection{The excursion associated to a slot diagram}
\label{sde2}

qOur goal is to build a bijection between walks and point configurations. 
We begin by showing how to construct the excursion associated to a slot diagram. 

In order to obtain the diagram for a single excursion we may assume that $\eps(0)=0$, or otherwise shift it so that this condition is satisfied, and extend the path outside the excursion with $\xi'=-1$. Let us denote by $\cE_o(n)$ the set of excursions starting at the origin, $\so(\vep)=0$, that have exactly $n$ maxima. We codify $\vep\in\cE_o(n)$ via the alternating minima and maxima, $\vep=(\ua_n, \ub_{n})=(a_1,\dots, a_n, b_1,\dots, b_n)$, so that
\begin{align}\label{Rdef}
  \cE_o(n)&\vcentcolon =\Bigl\{(\ua_n, \ub_n)\,: 0=a_1<b_1<a_2<\dots<a_n<b_n,\\
  &\qquad\qquad\qquad\,\sum_{j=1}^k(b_j-a_{j})> \sum_{j=1}^k(a_{j+1}-b_{j}),\, 1\le k<n-1\Bigr\}. \notag
\end{align}
The endpoint of the excursion satisfies $ \se(\vep)=2\sum_{j=1}^n(b_j-a_{j})$.

We will denote the set of excursions $\vep$ starting at the origin and such that $\vep(0)=0$ by
\begin{align}
 \cE_o&\vcentcolon = \cup_{n\ge1} \cE_o(n), \label{cEo}, 
\end{align}
and, as before, use $A[\vep]$ and $B[\vep]$ to denote the sets of minima and maxima of $\vep$.

The set of slot diagrams is denoted by
 \begin{align}
 \cN_o&\vcentcolon = \Bigl\{\sM\subset \RR\times \RR_+: \#\sM<\infty; \, 
   0\le u \le 2\sum_{(\cku, \ckl)\in\sM}(\ckl-\ell)^+, \, (u, \ell)\in\sM\Bigr\};\label{N0}\\
  \cN_o(n)&\vcentcolon = \Bigl\{\sM\in\cN_o: \#\sM=n\Bigr\}.\label{Nn2}
\end{align}
Denoting $\ell_{\max}\vcentcolon =\max\{\ell:(u, \ell)\in\sM\}$, we have $(0, \ell_{\max})\in\sM$ for $\sM\in\cN_o$.

\begin{lemma}[Space of slot diagrams]
 \label{ssd1}
 The set of slot diagrams of the excursions starting at $0$ is contained in $\cN_o$,
 \begin{align}
 \label{en1}
 \{\sN[\vep]:\vep\in\cE_o\}\subset \cN_o.
 \end{align}
\end{lemma}

\begin{proof}
Let $(0,\se]$ be the domain of the excursion.
Given $b\in B[\eps]$ we have $o_b\leq e$, so for $\ell=\ell_b$ it must hold that $h_\ell(0,o_{b})\leq h_\ell(0,\se)$.
By \eqref{eqhelloe} it follows that the point $(h_\ell(0,o_{b}),\ell)$ verifies the inequality in the definition of 
$\cN_o$, that is \eqref{N0}.
We conclude that $\sN[\vep]\in \cN_o$.
\end{proof}

Given a slot diagram $\sM\in \cN_o(n)$, let $\ell_{\max}\vcentcolon = \max \{\ell: (u,\ell)\in \sM\}$, and $\sg=\sg[\sM]$ be the function $\sg:\R_+\to\R_+$ defined as
\begin{align} 
\sg(y)&\vcentcolon = 2\sum_{(u, \ell)\in\sM} (\ell-y)^+, \quad 0\le y\le \ell_{\max},\label{sca7}
\end{align}
the total $y$-slot space in the slot diagram $\sM$. 

Fig.~\ref{inverse44} depicts the slot space created by the solitons of the excursion in Fig.~\ref{inverse1}.
For example, $\sg(\ell_2)$ is the length of the red segment.
Each maximum $b$ in the excursion $\vep$ contributes with $2(\ell_b-\ell)^+$ additional $\ell$-slot space; the union over $\ell<\ell_b$ of the $\ell$-slot spaces forms a triangle of height $\ell_b$ and base $2\ell_b$; the picture shows the triangles associated to the 3 solitons in the excursion. The soliton associated to $b_{\max}$ is mapped to $(0, \ell_{\max})$ and each remaining point $(u, \ell)\in\sN[\vep]$ belongs to the disjoint union of triangles with heights larger than $\ell$. The upper boundary of the region in Fig.~\ref{inverse44} if given by the graph of $\sg$.

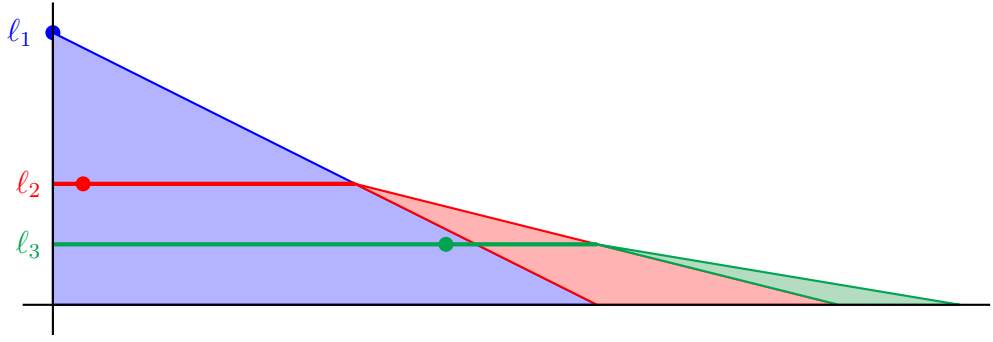
\begin{figure}
\centering
\begin{tikzpicture}[scale=0.4, thick]

	\node[dot, blue,label={[blue]left:{$\ell_1$}}] at (0,9) {};	
	\draw [blue, fill=blue!30] (0,9) -- (18,0) -- (0,0) -- cycle;

	\node[red,left] at (0,4) {$\ell_2$};    
    \node[dot, red] at (1,4) {};	
    \draw[red, ultra thick] (0,4) -- (10,4) ;
    \draw [red, fill=red!30] (10,4) -- (26,0) -- (18,0) -- cycle;

    \node[Green,left] at (0,2) {$\ell_3$};    
    \node[dot, Green] at (13,2) {};	
    \draw[Green, ultra thick] (0,2) -- (18,2) ;
    \draw [Green, fill=Green!30] (18,2) -- (30,0) -- (26,0) -- cycle;

	\draw (0,-1) -- (0,10);
	\draw (-1,0) -- (31,0);

\end{tikzpicture}
\caption{Slot diagram of the excursion in Fig.~\ref{inverse1}, and the slot space contributed by each soliton. }\label{inverse44}
\end{figure}

\begin{definition}
 [Mapping slot diagrams to excursions] \label{17}
\rm Given $\sM\in\cN_o(n)$, order its points by decreasing $\ell$-th coordinate, $\sM=\{(u_1, \ell_1), \dots, (u_n, \ell_n)\}$, with $\ell_{\max}=\ell_1\geq\dots\geq \ell_n$, and $u_1=0$. We construct an excursion $\vep[\sM]$ iteratively.

First step. Insert an $\ell_1$-soliton $\gamma_1$ at the origin to obtain an excursion $\vep^1[\sM]$ with one maximum $b_1=\ell_1$ and a minimum $a_1=0$. The left endpoint of $\gamma_1$ is denoted $o_1$, and in this case it matches $a_1$, $0_1=a_1$.
Fig.~\ref{e1} illustrates $\vep^1[\sM]$.

{
\centering
 \includegraphics
 [width=.80\textwidth, trim= 0 10mm 0 4mm, clip]
 {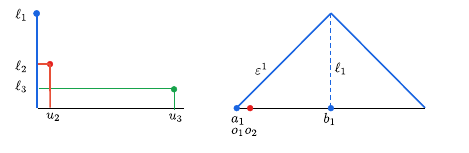}
 \captionof{figure}{Left: slot diagram $\sM$ with 3 points. Right: $\vep^1[\sM]$ obtained by inserting the blue soliton with maximal size $\ell_1$ at the origin. The red point $o_2$ is the insertion point of the second soliton. } \label{e1}
}
Recursive step. Suppose that for $k<n$ we have constructed an excursion $\vep^k$ with non necessarily ordered minima $A^k=\{a^k_1, \dots, a^k_{k}\}$ and maxima $B^k=\{b^k_1, \dots, b^k_{k}\}$, associated to points $(u_1,\ell_1),\dots, (u_k,\ell_k)$ in $\sM$ with decreasing $\ell$-th coordinate.

Denote by $(u,\ell)=(u_{k+1}, \ell_{k+1})$, and $h^k_\ell\vcentcolon = h_\ell[\vep^k]$, see \eqref{dhs1} for the definition of the latter. We insert a new soliton of height $\ell$ at the point $\so_{k+1}$ that is the solution $\so$ of  $h^k_\ell(0,\so) =u$. This amounts to shifting  the current extrema that fall to the right of $\so$ by $2\ell$ and adding a new maximum at $\so+\ell$ and a new minimum at $\so$ or $\so+2\ell$, in such a way that the resulting extrema are consistent with a path. More precisely,
\begin{gather}
  a^{k+1}_i \vcentcolon = a^k_i + 2\ell\, \one\{a^k_i>\so\},\quad
  b^{k+1}_i \vcentcolon = b^k_i + 2\ell\, \one\{b^k_i>\so\},\quad i\le  k,\\[2mm]
  a^{k+1}_{k+1}\vcentcolon = \so + 2\ell \,\one\{(\vep^k)'(\so)=1\}, \quad
  b^{k+1}_{k+1} \vcentcolon = \so+ \ell.
\end{gather}
Here $(\vep^k)'(\so)$ is the derivative of the path $\vep^k$ at the point $\so$. If $(\vep^k)'(\so)=1$, then $o$ lies between a minimum $a\in A^k$ and a maximum $b\in B^k$, and then the new minimum $a^{k+1}_{k+1}=\so+2\ell$; this is the case of the red soliton in Fig.~\ref{e2}. Otherwise, $(\vep^k)'(\so)=-1$, and the minimum associated to $b$ is placed at the left endpoint of the interval, $a^{k+1}_{k+1}=o$; this is the case of the green soliton in Fig.~\ref{e3}. 

Define $\vep^{k+1}$ as the unique excursion with extrema $A^{k+1},B^{k+1}$ and $\vep[\sM]\vcentcolon = \vep^n$.

The excursion in Fig.\/\ref{e2} is $\vep^2[\sM]$ for $\sM$ in Fig.~\/\ref{e1}, while $\vep^3$ is depicted in Fig.~\/\ref{e3}.
\end{definition}

{
\centering
 \includegraphics
 [width=.9\textwidth, trim={5mm 5mm 0 6mm}, clip ]
 {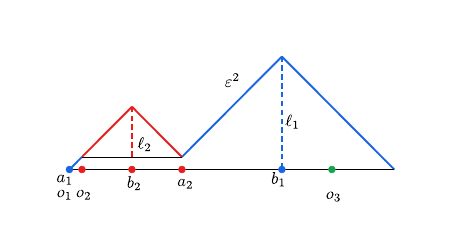}
 \captionof{figure}{Excursion $\vep^2$, obtained by placing a maximum $b_2=o_2+\ell_2$ and a minimum $a_2=o_2+2\ell_2$, and shifting the graph  of $\vep^1$ that lies to the right of $o_2$ by $2\ell_2$. The green dot is the bearer $o_3$ of the next soliton. We have dropped the superindices on the extrema.} \label{e2}
}
{
\centering
 \includegraphics
 [width=.81\textwidth, trim={0 0 0 6mm}, clip ]
 {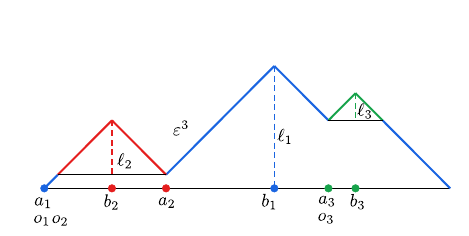}
 \captionof{figure}{Excursion $\vep^3$, obtained from $\vep^2$ by placing a maximum $b_3=o_3+\ell_3$ and a minimum $a_3=o_3$, and shifting the graph of $\vep^2$ that lies to the right of $o_3$ by $2\ell_3$. We have dropped the superindices on the extrema.} \label{e3}
}

When reconstructing the excursion, we add the solitons in decreasing order of height, starting from the largest. During the identification of solitons in Subsection \ref{subsecsol}, on the other hand, we identify them in increasing order of height, from the smallest to the largest, as shown in \eqref{orderedell}.
The intermediate paths considered in both processes are identical. We have

\begin{proposition}
  \label{p35}
  The map $\sM\mapsto\vep[\sM]$ is the inverse of the map $\vep\mapsto\sN[\vep]$,
 \begin{align} \label{inv9}
        \cE_o  & \longleftrightarrow    \cN_o \\
        \vep &\longmapsto \sN[\vep] \notag \\
        \vep[\sM] &\mapsfrom \sM \notag
\end{align}
\end{proposition}

\subsection{From point configurations to walks}
\label{ipc1}

We now show how to build a walk from a point configuration by joining a chain of excursions. 

Recall the definition of the set $\cN$ of $\mathcal{T}$-finite point configurations, \eqref{7}. We construct a map from $\cN$ to  $\hcW$, the space of walks with a record at the origin.
The first step is to partition $\sN\in\cN$ into subsets associated to slot diagrams $\sM_i\in\cN_o$. We recover the walk by concatenating the excursions $\vep(\sM_i)$ and inserting intervals of records between them.

Let $\sN\in\cN$, $\sN=\sN_-\cup\sN_+$ with
\begin{align}
 \sN_+ &\vcentcolon =\{(u, \ell)\in\sN: u\ge0\}, & \sN_- &\vcentcolon =\{(u, \ell)\in\sN: u<0\}, \label{sn+-}\\
 \intertext{and define the sets of point configurations}
  \cN_+&\vcentcolon = \{\sN_+:\sN\in\cN\},&\cN_-&\vcentcolon = \{\sN_-:\sN\in\cN\}.\label{cn+-}
\end{align}

\begin{definition}
  [The leftmost slot diagram of a positive point configuration]  \label{18}
 
\rm Given $\sN\in\cN_+$ we define the leftmost slot diagram $\sM$ contained in $\sN$ by means of a recursive argument.
Denote $\su=\su[\sN]$ the horizontal coordinate of the leftmost point in $\sN$,
\begin{align}
 \label{sca1}
 \su\vcentcolon = \min\bigl\{u:(u, \ell)\in\sN\bigr\};\qquad (u_1, \ell_1)\vcentcolon =\text{ point realizing the minimum}.
\end{align}
If there is more than one point that realizes the minimum we pick the one with the larger $\ell$. Note that $\su=u_1$, for clarity in the inductive step we keep both notations.
Since $\cN$ is $\mathcal T$-finite, the set $\cN\cap \{(u,\ell): u\leq x \}$ is finite for every $x>0$ and therefore there is a minimum.

First step. Let $\sM^1= \{(u_1,\ell_1)\}$.

Recursive step.  Assume we have defined $\sM^k=\{(u_1, \ell_1), \dots, (u_k, \ell_k)\}$.
If there is no point $(\cku, \ckl)\in\sN\setminus \sM^k$ satisfying 
\begin{align}
\cku-\su\le \tsum_{i=1}^k2(\ell_i-\ckl),
 \label{sca3}
\end{align}
then we stop and define the slot diagram $\sM\in\cN_o(k)$ by
\begin{align}
 \label{sca4}
 \sM\vcentcolon =\{(u_1-\su, \ell_1), \dots, (u_k-\su, \ell_k)\}.
\end{align}
Otherwise, define $\sM^{k+1}\vcentcolon = \sM^k\cup\{(u_{k+1}, \ell_{k+1})\}$, where $(u_{k+1}, \ell_{k+1})$ is the point satisfying \eqref{sca3} with maximal $\ell$. Iterate the recursive step.

Since $\sN$ is $\mathcal T$-finite, the iteration ends in a finite number of steps.
In fact, we can see inductively that \eqref{sca3} implies that every point in $\sM^k$ belongs to $\mathcal T(\su)$, which is finite.

When the algorithm stops, we have defined the map from $\cN_+$ to $\R_+\times\cN_o$ given by
\begin{align}
 \sN\mapsto (\su[\sN], \sM[\sN])\vcentcolon = \text{as defined in \eqref{sca1} and \eqref{sca4}}.
 \label{sc12}
\end{align}
$\sM[\sN]$ is called the leftmost slot diagram in $\sN$, with insertion point~$\su[\sN]$.  
\end{definition}

\begin{definition}
  [The slot diagrams of a positive point configuration]\label{d1}
 \rm  Let $\sN\in\cN_+$. Proceed recursively. Denote $\sN_0 \vcentcolon =\sN$, $\su_0\vcentcolon = 0$, $\sM_0\vcentcolon = \emptyset$, and a function $\sg_0(y) \equiv 0$ for all $y> 0$. 

Recursive step. Suppose that $\sN_{k},\su_{k},\sM_{k}$ are defined.
See Fig.~\ref{Mk}.
 For $k\ge1$ let $\sg_k\vcentcolon =\sg[\sM_k]$,  the right boundary of the slot diagram $\sM_k$, as defined in \eqref{sca7}. Then, use \eqref{sc12} to define
\begin{align}
 (\su_{k+1}, \sM_{k+1})&\vcentcolon = (\su[\sN_{k}], \sM[\sN_{k}]), \label{an1} \\
 \sN_{k+1} &\vcentcolon = \Bigl\{\bigl(u-\su_{k+1}-\sg_{k+1}(\ell), \ell\bigl):(u, \ell)\in\sN_{k}\setminus\sM_{k+1}\Bigr\}. \label{sc111}
\end{align}
A point $(u,\ell)$ in $\sN_{k}\setminus\sM_{k+1}$ is shifted leftwards to a point in $\sN_{k+1}$ having the same vertical coordinate~$\ell$, and horizontal coordinate $u-\su_{k+1}-\sg_{k+1}(\ell)$. 
Note that given a finite set $\mathcal T(y+\su_{k+1})$ that verifies \eqref{eqRy} for $\sN_{k}$, once we shift it as in \eqref{sc111} we get a finite set $\mathcal T(y)$ that verifies \eqref{eqRy} for $\sN_{k+1}$.
Therefore we inductively get that $\sN_{k}$ is $\mathcal T$-finite.

Iterate the recursive step to obtain a map from $\sN\in \cN_+$ to a sequence of slot diagrams and insertion points 
\begin{align}
  \sN\mapsto  (\su_k[\sN], \sM_k[\sN])_{k\ge 1}.\label{250}
\end{align}
\end{definition}

\begin{figure}

    \begin{subfigure}[b]{\textwidth}
\centering
\begin{tikzpicture}[scale=0.2, thick]

    \draw[ultra thick] (0,0) -- (2,-2);
    
\begin{scope}[shift={(2,-2)}]
    \draw[blue, ultra thick] (0,0) -- (5,5) -- (9,1) -- (17,9) -- (21,5) -- (23,7) -- (30,0);
    \draw[red, ultra thick] (1,1) -- (5,5) -- (9,1) ;
    \draw[Green, ultra thick] (21,5) -- (23,7) -- (25,5);
\end{scope}

    \draw[ultra thick] (32,-2) -- (35,-5);
	\draw[orange,ultra thick] (35,-5) -- (40,0)-- (44,-4) -- (46,-2) -- (47,-3) -- (49,-1)-- (53,-5);  
	\draw[violet,ultra thick] (44,-4) -- (46,-2) -- (47,-3) -- (49,-1)-- (52,-4);   
	\draw[yellow,ultra thick] (45,-3) -- (46,-2) -- (47,-3);    
	\draw[ultra thick] (53,-5) -- (55,-7);

\end{tikzpicture}

    \caption{A walk with identified solitons.}
    \end{subfigure}

    \begin{subfigure}[b]{\textwidth}
\centering
\begin{tikzpicture}[scale=0.2, thick]

\draw (-1,0) -- (55,0);
\draw (0,-1) -- (0,10);

\begin{scope}[shift={(2,0)}]
	\node[dot, blue] at (0,9) {};	
	\node[dot, red] at (1,4) {};	
	\node[dot, Green] at (13,2) {};
	 
	 \draw (0,9) -- (10,4) -- (18,2) -- (30,0);	
	 
	 \begin{scope}[shift={(3,0)}]
	 	\node[dot, orange] at (8,5) {};
	 	\node[dot, violet] at (17,3) {};
	 	\node[dot, yellow] at (32,1) {};
	 \end{scope}	   
\end{scope}	    

\end{tikzpicture}

    \caption{The point configuration with the delimitation corresponding to $\sM_1$ mark in black.}
    \end{subfigure}

    \begin{subfigure}[b]{\textwidth}
\centering
\begin{tikzpicture}[scale=0.2, thick]

\draw (-1,0) -- (55,0);
\draw (0,-1) -- (0,10);

\begin{scope}[shift={(3,0)}]
\draw (0,5) -- (4,3) -- (12,1) -- (18,0);	
	 	\node[dot, orange] at (0,5) {};
	 	\node[dot, violet] at (3,3) {};
	 	\node[dot, yellow] at (8,1) {};
\end{scope}	    

\end{tikzpicture}

    \caption{$\sN_2$ with the delimitation corresponding to $\sM_2$ mark in black.}
    \end{subfigure}

\caption{We show in the construction of the slot diagrams $\sM_k$ corresponding to a walk.}
\label{Mk}
\end{figure}
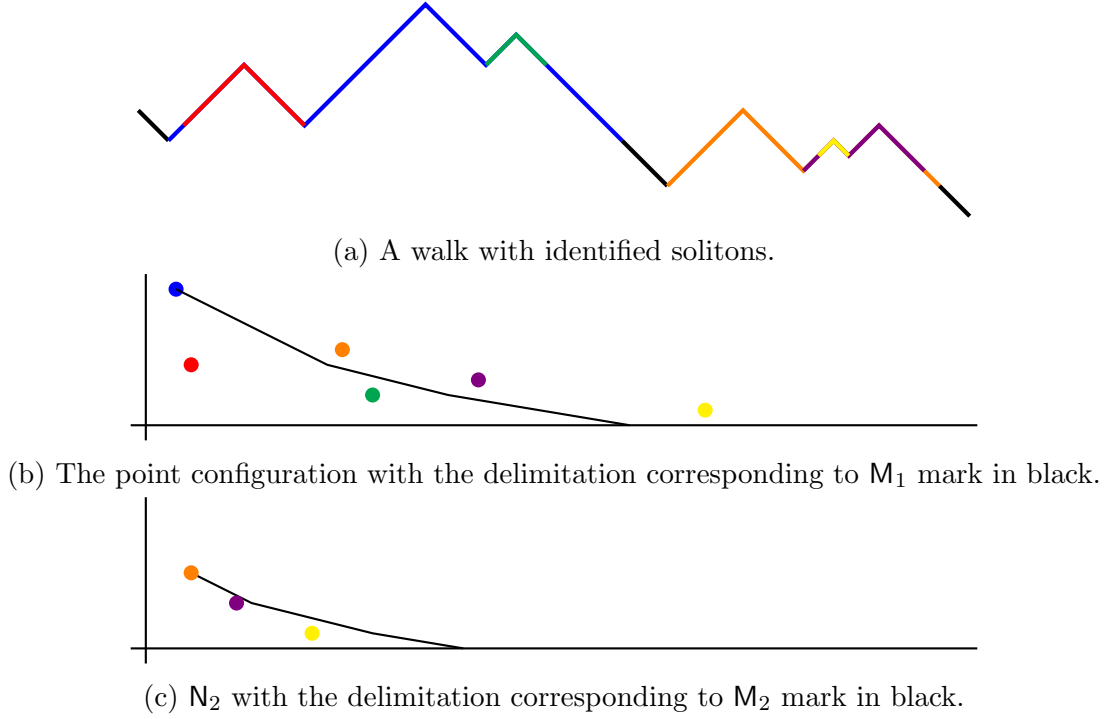

\begin{remark}
\label{equivTfin}\rm
We have defined $\cN$ as the set of point configurations that are $\mathcal T$-finite.
That condition implies that the construction of $\sM_k$ in Definition~\ref{d1} ends in a finite number of steps.
The points in $\mathcal T(y)$ are covered in a finite number of steps, let say $n$, at that point we get $\sum_{k=1}^n\su_k>y$. So $\sum_{k}\su_k=+\infty$.

Conversely, if the construction of $\sM_k$ ends in a finite number of steps for each $k$ and $\sum_{k}\su_k=+\infty$ we get that 
$\cN$ is $\mathcal T$-finite.
\end{remark}

\begin{definition}
  [Mapping  $\cN$ to a carrier process]\label{d8}\rm
Let $\sN\in\cN$ with $\sN=\sN_+\cup \sN_-$. Recall the construction of the excursion $\vep[\sM]$ from $\sM\in\cN_o$ in Definition \ref{17}, and let 
$\tau(\vep)$ be the length of the excursion $\vep$. We have $\tau(\vep)=\se(\vep)=\se(\vep)-\so(\vep)$ as $\so(\vep)=0$ for $\vep \in \cN_0$. Consider the sequence $(\su_k,\sM_k)_{k\ge1}$ obtained in \eqref{250} from $\sN_+\in\cN_+$. Denote $\so_0\vcentcolon =0$, $\se_0\vcentcolon =0$, and for $k\ge1$, 
\begin{align}
 \vep_k&\vcentcolon =\vep[\sM_k], \quad \se_k\vcentcolon =\se(\vep_k),\label{vep3}\\
 \so_k&\vcentcolon = \so_{k-1}+ \se_{k-1}+ \su_k, \quad k\ge1.\label{ak3}
\end{align}
For $\sN_+\in\cN_+$, we define the carrier process $\zeta_+=\zeta[\sN_+]:\RR_+\to\RR_+$ by concatenating excursions $\vep_k$ separated by intervals of lengths $\su_k$, as follows
\begin{gather}
 \label{mex1}
 \zeta(x) \vcentcolon =
 \begin{cases}
 \vep_k(x-\so_k), 
 &\so_k< x\le \so_k+\se_k, \\
  0, 
 &\so_{k-1}+\se_{k-1}< x \le \so_{k}, 
 \end{cases},\qquad k\ge1.
\end{gather}

The reflection with respect to the vertical axis of $\sN_-\in\cN_-$ is given by
\begin{align}
 \label{rfl1}
  \Reflect(\sN_-) &\vcentcolon = \{(-x, \ell):(x, \ell)\in\sN_-\} \in\cN_+.
\end{align}
For $\sN_-\in \cN_-$, we define $\zeta_-=\zeta[\sN_-]:\RR_-\to\RR_+$ by reflecting the previous construction, 
\begin{align}
 \zeta[\sN_-] &\vcentcolon = \Reflect\bigl[\zeta[\Reflect(\sN_-)]\bigr],  \label{257}
 \\
\text{where}\quad\Reflect\zeta(x)&\vcentcolon = \zeta(-x). \notag
\end{align}
The mapping from $\sN\in \cN$ to the carrier process $\zeta=\zeta[\sN]$ is given by
\begin{align}
  \label{eq:5}
  \zeta[\sN](x) =
  \begin{cases}
    \zeta[\sN_+](x),&x\ge 0\\
    \zeta[\sN_-](x),&x<0.
  \end{cases}
\end{align}
\end{definition}

\begin{definition}
  [Mapping $\cN$ to $\cW$]\label{d9}
  To construct the walk $\xi[\sN]$ we map $\sN\in\cN$ to the carrier $\zeta= \zeta[\sN]$ using Definition \ref{d8} and then $\xi[\sN]=\xi[\zeta]$ using the mapping \eqref{4}.
\end{definition}

Given $\xi\in \hcW$ and $\sN[\xi]$ the associated point process, the set $\sM_i[\sN[\xi]]$ is the slot diagram of $\vep_i[\xi]$, the $i$-th excursion of $\xi$. The value $\su_i$  is the length of the interval of records between the $(i-1)$-th and $i$-th excursions. See Fig.~\/\ref{3}.
  By Proposition~\ref{p35} we obtain the following theorem. 

\begin{theorem}
The following functions
    \begin{alignat*}{2}
        \cN  & \longleftrightarrow    \hcW  \\
        \sN &\longmapsto \xi[\sN] \\
        \sN[\xi] &\mapsfrom \xi \\
    \end{alignat*}
are inverse of each other.
\end{theorem}

{	\centering
	
  \includegraphics[width=.7\textwidth]{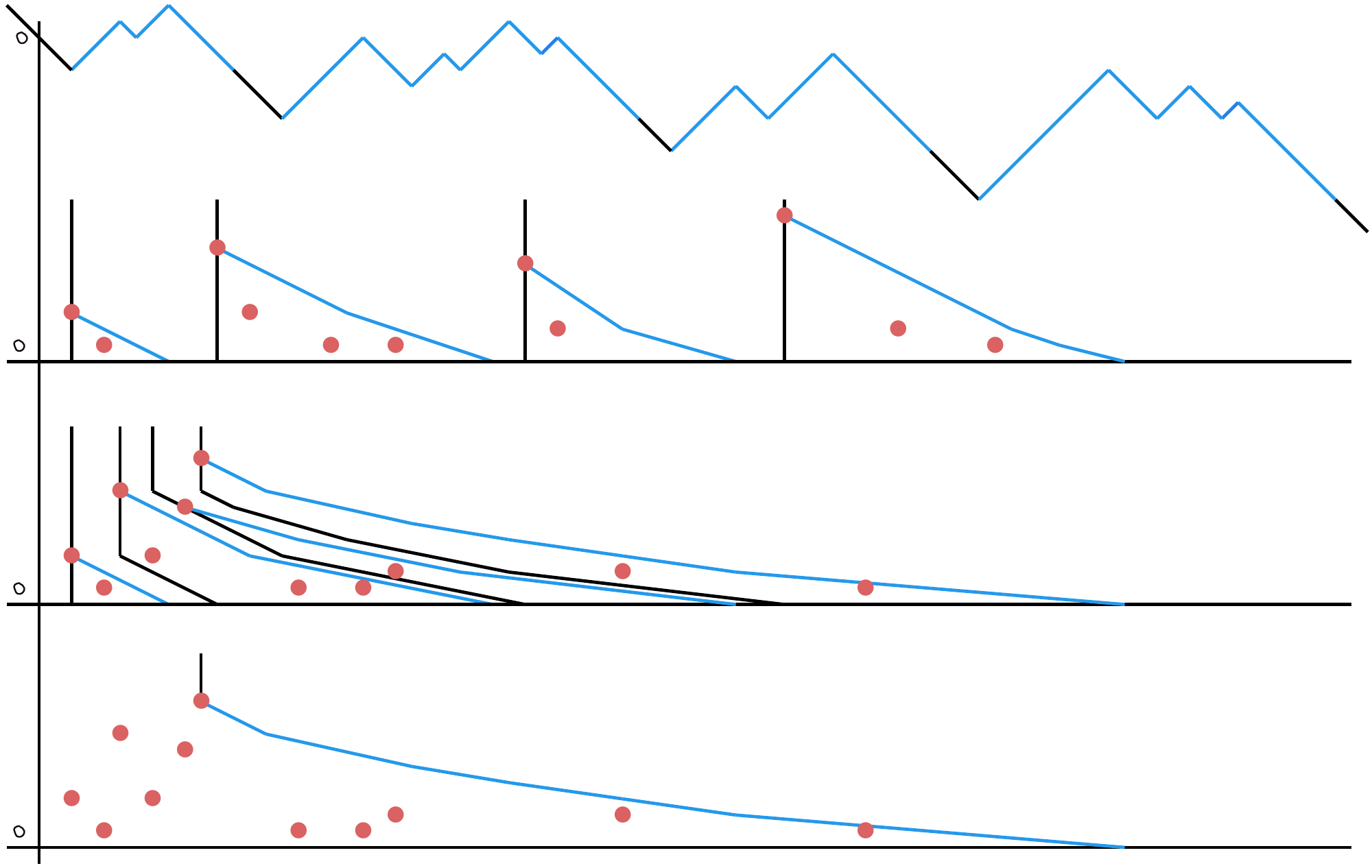}
  \captionof{figure}{From walks to point configurations. \label{3}}
  \par
}

\section{Random distributions of walks}
\label{s40}

In this section we introduce two families of measures on trajectories. The first family consists of product measures on the decomposition of the walk into excursions, which are separated by iid intervals of records. The distribution of an excursion $\vep$ is determined by a vector of weights $(\alpha(\ell_i))_{i}$, where $(\ell_i)_i$ are the heights of the solitons in $\vep$ and the intensity $\alpha:\RR_+\to \RR_+$ satisfies some integrability conditions. We then consider another family of distributions determined by Poisson point processes with configurations in $\cN$. We finally show that under some conditions these families of measures are equivalent.

\subsection{Random excursions}
\label{2.6}

We start by defining two families of distributions on excursions in $\cE_o$.

\subsubsection{Soliton weights}
\label{2.6.1}

Recall that $\cE_o(n)$ denotes the set of excursions starting at the origin and having $n$ maxima, as defined in \eqref{Rdef}.
For $\alpha:\RR_+\to \RR_+$, define
\begin{align}
  \label{za3}
  Z_\alpha&\vcentcolon = \sum_{n=1}^\infty\int_{\cE_o(n)}\alpha(\ell_1)\dots\alpha(\ell_n)\, d \ua_n\, d\ub_{n},
\end{align}
where $\ell_1\geq \ell_2\geq  \dots \geq \ell_n$ are the soliton sizes of the excursion $(\ua_n, \ub_{n})\in\cE_o(n)$,  $d\ua_n \vcentcolon = da_2\dots da_n$, because $a_1\equiv0$, and $ d\ub_{n}\vcentcolon =db_1\dots db_n$. Define
\begin{align}\label{mua0}
  	\mathcal A&\vcentcolon =\bigl\{\alpha:\RR_+\to \RR_+\,:\, Z_\alpha<+\infty\bigr\}\,.
\end{align}
For $\alpha \in\cA$,
define the probability measure $\mu_\alpha$ on $\cE_o$ by
\begin{align}
	\label{mua1}
	\mu_\alpha(d\varepsilon)
  &\vcentcolon =\frac{1}{Z_\alpha}\, \sum_{n=1}^\infty\one\{\vep\in\cE_o(n)\}\, \alpha(\ell_1)\dots\alpha(\ell_n)\, d\ua_n\, d\ub_{n}\,.
    \end{align}
The length of $\vep\in\cE_o(n)$ is $\se(\vep)=2\sum_{i=1}^{n}\ell_i$, as $\so=0$, see \eqref{1}. The average length of an excursion distributed as \eqref{mua1} is $\int \se(\varepsilon)\,\mu_\alpha(d\varepsilon)$. We call
\begin{equation}
	\label{mua2}
	\mathcal A^+\vcentcolon =\Bigl\{\alpha\,:\,\int \se(\varepsilon)\,\mu_\alpha(d\varepsilon)< +\infty\Bigr\}\,
\end{equation}
the set of weight functions that determine distributions with finite average excursion length.
It is hard to characterize the families $\mathcal A, \mathcal A^+$ in terms of soliton weights as in (\ref{mua1},~\ref{mua2}), but we will see that they can be easily described using the slot representation of the excursions. 

\begin{example}
  [Excursions of a telegraph process]\label{teleg}\rm
Let $(\eta(t))_{t\ge0}$ be a continuous time Markov jump process on $\{-1, 1\}$ 
with rates $r(-1, 1)=\lambda_-\le \lambda_+=r(1, -1)$. The Kac telegraph process \cite{zbMATH03491985} describes a one dimensional particle travelling with constant speed in the direction given by $\eta$. Let us denote this process by $(\xi(t),\,t\ge 0)$. The typical excursion $\vep\in\cE_o$ starts with $\eta(0)=1$ and runs until the first time $\xi$ hits $0$. We get
\begin{align}
	\xi(t)=\int_0^t \eta(s)\, ds,\quad
	\vep =\bigl(\xi(t)\bigr)_{t\in[0, \se]}, \quad \text{where} \quad \se = \inf\{t>0: \xi(t)=0\}\,.\label{6}
\end{align}

We can explicitly compute the distribution of the excursions if we use that the stretches where $\xi$ has constant direction $\pm 1$ are independent exponential variables with parameter $\lambda_{\pm}$. It follows that the distribution of $\vep\in\cE_o$ is absolutely continuous with respect to the Lebesgue measure in $\cup_{n\ge1}\RR^{2n-1}_+$, with density $f(\vep)=\sum_{n\ge1}f_n(\vep)\one_{\cE_o(n)}(\vep)$, where for $\vep=(\ua_n,\ub_n)\in\cE_o(n)$ we have
\begin{align}
  f_n(\ua_n,\ub_n)
  &=
    \Bigl(\prod_{i=1}^n \lambda_+\,e^{ -(b_i-a_i)\lambda_+}\Bigr)\,
    \Bigl(\prod_{i=1}^{n-1} \lambda_-\,e^{-(a_{i+1}-b_i)\lambda_-}\Bigr)\,
    e^{ - (\se-b_n)\lambda_-} \\
  &= \Bigl(\prod_{i=1}^n \lambda_+\,e^{-\ell_i\lambda_+} \Bigr)\,
    \Bigl(\prod_{i=1}^{n} \lambda_-\,e^{-\ell_i\lambda_-}\Bigr)\,\frac{1}{\lambda_-}
  \label{eq:6}\\
&=  \frac{1}{\lambda_-} \prod_{i=1}^n \alpha(\ell_i).
\end{align}
We used \eqref{eq:4} to obtain \eqref{eq:6} from the first line. Here $\alpha:\RR_+\to \RR_+$ is given by
\begin{equation}\label{zza1}
	\alpha(\ell)=\lambda_-\, \lambda_+\, e^{-\left(\lambda_-+\lambda_+\right)\ell}.
\end{equation}
Since $\lambda_-\le \lambda_+$ the random excursion $\vep$ is finite with probability one. If $\lambda_- < \lambda_+$ the process has  negative drift, and the excursion has finite mean. In any of these cases we have
\[
  \sum_{n\ge0}\int_{\cE_o(n)}f_n(\ua_n,\ub_n)\, d\ua_nd\ub_n =P(\vep\in\cE_o)=1.
\]
Hence $Z_\alpha = \lambda_-$,
and the telegraph process has excursions with
law $\mu_\alpha$, with $\alpha \in \mathcal A$ as in \eqref{zza1}. If the strict inequality holds, $\lambda_-<\lambda_+$, then $\alpha \in \mathcal A^+$.

We can construct the telegraph process as seen from a record by considering a sequence of iid excursions $(\vep_k)_{k\ge 1}$ with law $\mu_{\alpha}$, inserting an interval of length $\su_k$ between the $k-1$-th and the $k$-th excursion, $k\ge 2$, and adding an interval of length $\su_1$ between the record and the first excursion, see \eqref{mex1}-\eqref{eq:5}. The variables $(\su_k)_{k\ge 1}$ here are iid exponential random variables with parameter $\lambda_-$. This yields the carrier process $\zeta(t)$, in order to obtain the telegraph process $\xi(t)$ we use \eqref{4}. 
\end{example}

\subsubsection{First excursion of a Poisson point process}
\label{2.6.2}

We introduce a second family of distributions on $\cE_o$, which rests on the representation of an excursion as a slot diagram.

Let
\begin{align}	\mathscr Q&\vcentcolon =\Big\{q:\RR_+\to \RR_+ \text{ measurable}:\, \int_0^{+\infty}q(\ell)d\ell <+\infty\Big\}, \label{Q1}\\
  \mathscr Q^+&\vcentcolon =\Big\{q\in \mathscr Q\,:\, \int_0^{+\infty}\ell q(\ell) d\ell <+\infty\Big\}\,.\label{Q2}
\end{align}
Given $q\in \mathscr Q$, let $\sN$ be the Poisson point process on $\RR\times\RR_+$ with intensity measure $du\, q(\ell)d\ell$.

Denote by $\su$ is the distance of the leftmost point in $\sN_+$ to the $\ell$ axis, if such a leftmost point exists, and $0$ otherwise. Since $q\in \mathscr Q$, the law of $\su$ is exponential with rate 
\begin{equation}\label{L1}
\int q\vcentcolon =\int_0^\infty q(y)dy,\qquad \PP(\su>t) = e^{-t\int q},
\end{equation} 
and $\su$ is positive with probability $1$. Let $(\su[\sN_+], \sM[\sN_+])$ be as defined by the process \eqref{sc12}, where 
$\su[\sN_+]=\su$ and $\sM$ is the first slot diagram of $\sN_+$. At this point we cannot exclude the possibility that $\sM[\sN_+]$ is infinite, but the fact that $q\in \mathscr Q$ implies that the recursive step in the construction is well defined, and the number of points in $\sM\cap (\RR_+\times [a,b])$ is a.s. finite for any $0<a<b$. We show in Lemma~\ref{Tpos}  below that $\sN_+$ belongs to $\mathcal N$ and $\sM[\sN_+]$ is finite with probability $1$.

Note that $\su$ and $\sM$ are independent, as $\sN$ is a Poisson point process with an intensity that is homogeneous in the first coordinate. In particular, the coordinates of the first point $(\su, \ell_1)$ in the leftmost excursion are independent, and $\ell_1$ is sampled with law $q(\ell) d\ell/\int q$, a fact we isolate for future reference.

\begin{remark}(Law of the excursion maximum)\label{maxi}
The maximum value $\ell_1$ of the excursion $\sM$ has distribution
\begin{equation}\label{maxlaw}
 \frac{q(\ell)\,d\ell}{\int_0^{+\infty} q(\ell)d\ell}.
\end{equation}
\end{remark}

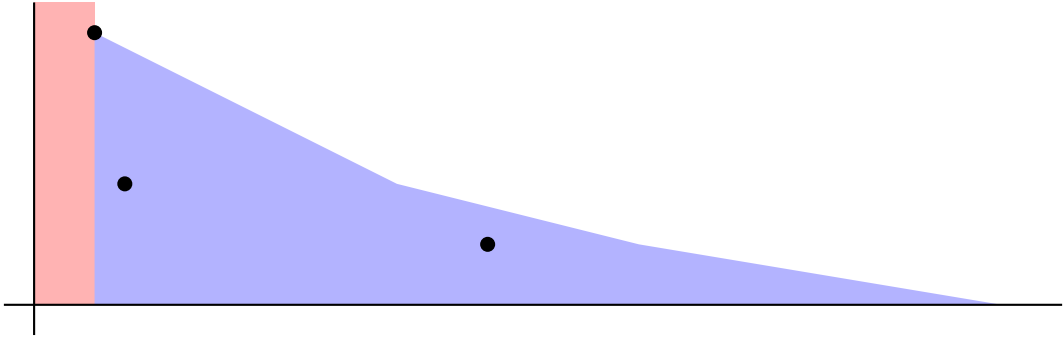
\begin{figure}[th]
\centering
\begin{tikzpicture}[scale=0.4, thick]

\draw [fill=red!30, draw=none] (0,0) -- (2,0) -- (2,10) -- (0,10) -- cycle;	

\begin{scope}[shift={(2,0)}]
	 \draw [fill=blue!30, draw=none] (0,9) -- (10,4) -- (18,2) -- (30,0) -- (0,0) -- cycle;	
	 	\node[dot] at (0,9) {};	
	\node[dot] at (1,4) {};	
	 \node[dot] at (13,2) {};
\end{scope}	    

\draw (-1,0) -- (34,0);
\draw (0,-1) -- (0,10);

\end{tikzpicture}
  
\caption{The regions $\sU$ (red) and $\sG$ (blue) for the excursion in Fig.~\ref{inverse1}.}
\label{UG}
\end{figure}

In order to identify the leftmost slot diagram $\sM$ of a configuration $\sN\in \mathcal N$, the process defined in \eqref{sc12} explores the regions 
\begin{align}
\sU&\vcentcolon =\{(z, y):0<z\le \su\}\label{sA1}\\
\sG&\vcentcolon =\{(z, y):\su\le z<\sg(y)\}\label{sA2},
\end{align}
and locates all the points in the diagram, see Fig.~\ref{UG}. Here $\sg=\sg[\sM]:\RR_+\to \RR_+$ is defined in \eqref{sca7}.

Recall the definitions \eqref{N0} and \eqref{Nn2} of the sets $\mathcal N_o$ and $\mathcal N_o(n)$. 
The Poisson process induces a measure $\nu_q$ on finite slot diagrams  
such that for $n\ge 1$ and $\sM = \{(u_1, \ell_1), \dots, (u_n, \ell_n)\} \in \mathcal N_o(n)$ 
\begin{align}
	\label{nuq4}
	\nu_q(d\uu_n\, d\uell_n) = \one_{\mathcal N_o}\left(\underline u_n, \underline \ell_n\right)\frac{q(\ell_1)\, d\ell_1}{\int q(y)dy}\, \exp\Bigl(-\int_{\sG} q(y)\, dxdy\Bigr)\, \prod_{i=2}^n q(\ell_i)\, du_i\, d\ell_i, 
\end{align}
where $d\uu_n=du_2\dots du_n$, as $\sM\in\cN_o$ implies that $u_1\equiv 0$, and $d\uell_n=d\ell_1\dots d\ell_n$. The indicator function $\one_{\mathcal N_o}\left(\underline u_n, \underline \ell_n\right)$  ensures that the coordinates determine a slot diagram. Proving that the first slot diagram determined by the Poisson point process is finite with probability $1$ is equivalent to showing that $\nu_q$ is a probability measure, see Lemma~\ref{Tpos} below.

The set $\sG$ in \eqref{sA2} is a disjoint union of triangles $(\sG_i)_i$, where $\sG_i$ has base $2 u_i$ and height $\ell_i$, see Fig.~\ref{inverse44}, hence
\begin{align}
	\label{gii7}
	\int_\sG q(y)\, dxdy = \sum_{i=1}^n \int_{\sG_i}dx\, q(y)\, dy=2\sum_{i=1}^n \int_0^{\ell_i} (\ell_i-y)\, q(y)\, dy.
\end{align}
We obtain the following factorization of $\nu_q$,
\begin{align}
	\label{nuq5}
	\nu_q(d\uu_n\, d\uell_n) &= \frac{\one_{\mathcal N_o}\left(\underline u_n, \underline \ell_n\right)}{\int q(y)dy}\, \prod_{i=1}^n \exp\Bigl(-2\int_0^{\ell_i} (\ell_i-y)\, q(y)\, dy\Bigr)\, q(\ell_i)\, d\uu_n\, d\uell_n \notag\\
	&= \frac{\one_{\mathcal N_o}\left(\underline u_n, \underline \ell_n\right)}{\int q(y)dy}\, \prod_{i=1}^n \exp\big(-2\mathcal Q(\ell_i)\big)\, \mathcal Q''(\ell_i)\, d\uu_n\, d\uell_n ,
\end{align}
where 
\begin{align}
	\label{qm7}
	Q(y)\vcentcolon =\int_0^yq(z)dz \quad \text{and} \quad \mathcal Q(y)\vcentcolon =\int_0^yQ(z)dz.
\end{align}

\subsubsection{Finite number of solitons and mean excursion length}\label{aux}

The goal of this section is to prove that for $q\in \mathscr Q$, the first excursion of the Poisson process with intensity $du\,q(\ell)d\ell$ contains a finite number of solitons, while if $q\in \mathscr Q^+$ it has finite mean length, Lemmas~\ref{Tpos} and Corollary \ref{mlength}. To do so we introduce a non homogeneous branching process related to the construction of the slot diagram.

Let $q\in \mathscr Q$. For each $\ell>0$ consider the distribution on $[0,\ell ]$ given by 
\begin{align}
	q_\ell(y)\vcentcolon =\frac{(\ell-y)q(y)\one_{\,0\leq y\leq \ell}}{\mathcal Q(\ell)}.
\end{align}
The process starts with one single node, the root, that has label $\ell>0$. If a node has label $\tilde \ell$, it produces a Poisson number of offspring with parameter $(2 \mathcal Q(\tilde \ell))$, which are assigned independent labels sampled according to $q_{\tilde \ell}$.

Let us call $p_\ell$ the probability that the process dies out. By conditioning on the number of offspring and their labels we obtain that this probability satisfies the following recursive equation
\begin{equation}\label{eqint}
	p_\ell=e^{-2\mathcal Q(\ell)}\sum_{k=0}^{+\infty}\frac{\big(2\mathcal Q(\ell)\int_0^\ell q_\ell(y)p_y dy\big)^k}{k!}=e^{2\mathcal Q(\ell)[\int_0^\ell q_\ell(y)p_y dy-1]}\,.
\end{equation}
Let now $P_\ell\vcentcolon =\int_0^\ell(\ell-y) q(y)p_ydy$. The integral equation \eqref{eqint} can be stated as the following differential problem
\begin{equation}\label{eqnPl}
	\left\{
	\begin{array}{l}
		P''_\ell=q(\ell)e^{2(P_\ell-\mathcal Q(\ell))},\\
		P_0=0,\\
		P'_0=0.
	\end{array}
	\right.
\end{equation}
This is a Cauchy problem with a unique solution, and by a direct verification it is possible to show that $P_\ell=\mathcal Q(\ell)$ is the unique solution, which leads to $p_\ell=1$. The branching process is therefore always subcritical. 

The number of nodes in this branching process is distributed as the number of points in the leftmost slot diagram $\sM$ of the Poisson process with intensity $du\,q(\ell) d\ell$, conditioned to the leftmost point having coordinate $\ell$ equal to the label of the initial node in the branching process. Indeed, we may construct the branching process by assigning to each point point $(u_i, \ell_i)$ in the slot diagram the label $\ell_i$ (its second coordinate), and identifying its progenie with the points of the process that fall in the triangle $\sG_i$. We have thus proved the following result.

\begin{lemma}[Finite slot diagrams]\label{Tpos}
Let $\sN$ be a Poisson point process on $\RR\times \RR_+$ with intensity measure $du\,q(\ell) d\ell$, and let $\sM$ be the leftmost slot diagram in $\sN_+$. Then $\sM$ contains a finite number of points with probability $1$.
\end{lemma}

We apply a similar argument to derive conditions on $q$ that ensure that the mean length of the leftmost excursion
\begin{equation}\label{3131}
m(\ell)\vcentcolon=\EE\big[\sum_{(u_i, \ell_i)\in \sM}2\ell_i\big| \ell_1=\ell\big]
\end{equation}
is finite, given that the height of leftmost point is $\ell$. By conditioning on the points in the triangle $\sG_1$ with base $2\ell$ and height $\ell$ determined by this leftmost point, or equivalently, on the first generation of the branching process process, we have that
\begin{equation}\label{reclun}
\begin{aligned}
m(\ell)&=2\ell + \EE\big[\sum_{(u_i, \ell_i)\in \sG_1} m(\ell_i)\big| \ell_1=\ell\big]\\
	  &=2\ell+ \int_0^\ell (\ell-y)q(y)m(y)dy,
\end{aligned}
\end{equation}
where the second line follows from Campbell's formula. We can transform this integral equation into a differential problem by setting $Y(\ell)\vcentcolon =\int_0^\ell(\ell-y)m(y)q(y)dy$.  Then \eqref{reclun} is equivalent to
\begin{equation}\label{Ydiff}
	\left\{
	\begin{array}{l}
		Y''(\ell)=2q(\ell)\big(\ell +Y(\ell)\big), \\
		Y(0)=Y'(0)=0.
	\end{array}
	\right.
\end{equation}
With the exception of certain values of $q$, there is not a general form for the solution of this Cauchy problem. On the other hand, we obtain the following criterion.

\begin{lemma}\label{lemma10}
	Let $q \in L^1(\RR_+)$ and $m$ the solution of \eqref{reclun}.
		Then $\int_0^{+\infty}m(y)q(y)dy<+\infty$ if and only if $\int_0^\infty yq(y)dy<+\infty$. 
\end{lemma}
\begin{proof}
Let us call $X(s)=Y'(s)=\int_0^s m(y) q(y)dy$. The statement of the lemma concerns the behavior at infinity of $X$. In terms of $X$, problem \eqref{Ydiff} becomes 
\begin{equation}\label{Zdiff}
	\left\{
	\begin{array}{l}
		X'(\ell)=2\ell q(\ell)\big(1 +\frac{1}{\ell}\int_0^\ell X(y)dy\big), \\
		X(0)=0.
	\end{array}
	\right.
\end{equation}
If $\int_0^{+\infty}yq(y)dy=+\infty$ then, since all the terms on the right hand side of the equation in \eqref{Zdiff} are nonnegative, we get
\[
X(\ell)=\int_0^\ell X'(y)dy\geq 2\int_0^\ell yq(y)dy\,,
\]
from where it follows that
\[
\int_0^{+\infty}m(y)q(y)dy=\lim_{\ell\to +\infty}X(\ell)\geq \int_0^{+\infty} yq(y)dy=+\infty\,.
\]
Suppose now that $\int_0^{+\infty}yq(y)dy<+\infty$. Since $X$ is monotone non decreasing  we have from \eqref{Zdiff}
\[
\frac{X'(\ell)}{1+X(\ell)}\leq 2\ell q(\ell)\,,
\]
that can be integrated and in the limit gives
\[
\int_0^{+\infty}m(y)q(y)dy=\lim_{\ell\to +\infty}X(\ell)\leq e^{\int_0^{+\infty} yq(y)dy}-1<+\infty\,.\qedhere
\]
\end{proof}

We now return to the Poisson process. By Remark \ref{maxi} the mean length of the first excursion is 
\begin{align}\label{3132}
\EE[m(\ell_1)] =\frac{\int_0^\infty m(\ell) q(\ell) d\ell}{\int_0^\infty q(\ell) d\ell},
\end{align}
and Lemma \ref{lemma10} implies the following result.
\begin{corollary}[Finite mean excursion length]\label{mlength}
Let $q\in \mathscr Q$. Consider the Poisson point process $\sN$ in $\RR\times \RR_+$ with intensity $du\, q(\ell)d\ell$, and let $\sM$ be the leftmost excursion of $\sN_+$. Then its mean length
\begin{align}\label{3133}
\EE\big[ \sum_{(u_i,\ell_i) \in \sM}2\ell_i\big]=\frac{\int_0^\infty m(\ell) q(\ell) d\ell}{\int_0^\infty q(\ell) d\ell}<+\infty \iff q\in \mathscr Q^+.
\end{align}
\end{corollary}

\subsection{Equivalence of measures}\label{equiv}

 	In Sections \ref{2.6.1} and \ref{2.6.2} we introduced two probability measures on the space $\cE_o$ of finite excursions. The first one, $\mu_{\alpha}$, is parametrized by a weight function $\alpha$, see \eqref{mua1}, while the second one is associated to a Poisson process with intensity determined by a function $q \in L^1(\RR_+)$, see \eqref{nuq4} and \eqref{nuq5}. We now show that these measures are equivalent by relating the functions $\alpha$ and $q$.
	
	 Given $q\in \mathscr Q$, let $\alpha:\RR_+\to \RR_+$ be given by
	\begin{align}
		\label{a=q}
		\alpha(\ell) \vcentcolon= q(\ell)\, \exp\Bigl(-2\int_0^\ell(\ell-y)\, q(y)\, dy\Bigr)\,, \qquad \ell\in \mathbb R^+\,.
	\end{align}
	With $\alpha$ defined as above, and $\mathcal Q$ given in \eqref{qm7}, we get that $\alpha$ and $\mathcal Q$ satisfy 
	\begin{align}
		\label{a=q3}
		q(y)=\mathcal Q''(y)& =\alpha(y) e^{2\mathcal Q(y)}.
	\end{align}
	By definition $\mathcal Q(0)=\mathcal Q'(0)=0$. So to get the inverse map of \eqref{a=q} we need to solve 
	\begin{equation}\label{lequazione1}
		\left\{
		\begin{array}{l}
			\mathcal Q''(y)=\alpha(y) e^{2\mathcal Q(y)}, \\
			\mathcal Q(0)=\mathcal Q'(0)=0,
		\end{array}
		\right.
	\end{equation}
	and compute $q(\ell)=\mathcal Q''(\ell)$.
	
	For simplicity we prove the following equivalence result assuming the global existence of the differential problem \eqref{lequazione1}. This assumption can be removed, it follows from the definition of the set $\mathcal A$ in \eqref{mua0}.

\begin{proposition}[Equivalence of measures]
	\label{eqm9}
	Consider the relation  between the functions $\alpha$ and $q$ given by equations \eqref{a=q} and the Cauchy problem \eqref{lequazione1}.
	This defines a bijection between $\mathcal A$ and $\mathscr Q$ and between $\mathcal A^+$ and $\mathscr Q^+$. Moreover if $\alpha$ and $q$ are in bijection then 
	$ \mu_\alpha=\nu_q$ and $Z(\alpha)=\int_{\RR_+} q(\ell)\, d\ell$.
\end{proposition}
\begin{proof}

	We first show that for each fixed $n$,
the variables $(\underline u_n,\underline \ell_n)$ and $(\underline a_n,\underline b_n)$
are related by a volume preserving linear transformation, i.e. we show that the Jacobian determinant of this transformation is $1$. 

We proceed by induction. First of all the Jacobian determinant of the transformation is one when $n=1$ since in that case there are just the variables $\ell_1$ and $b_1$ with $\ell_1=b_1$. We assume that the statement is true for $n$ and show that it is then true for $n+1$. 

Consider an excursion having $n$ solitons with lengths $\ell_1\geq \ell_2\geq \dots \geq \ell_{n}$. Let us call $(\underline u_n,\underline \ell_n)$ the coordinates in the first representation, and let $(\underline a_n, \underline b_n)$be  the corresponding variables in the other coordinate system. By inductive hypothesis we have
$(\underline a_n, \underline b_n)=A(\underline u_n, \underline \ell_n)$ where $A$ is a $(2n-1)\times (2n-1)$ matrix 
having $|A|=\pm 1$. Indeed the structure of the matrix $A$ changes for different regions of the parameters but not the value of the determinant.
We call $(A_{i,j})_{1\leq i,j\leq 2n-1}$ the elements of the matrix $A$.

In Fig.~\ref{transform} we present an excursion with $n=3$ red solitons to which we attach a new blue soliton. After attaching this soliton of length $\ell_{n+1}\leq \ell_n$, the updated coordinates become
$(\underline u_n, \underline \ell_n, u_{n+1}, \ell_{n+1})$. 
The slot diagram of the new excursion will be in correspondence with coordinates $(\tilde{\underline a}_n, \tilde{\underline b}_n, a', b')$, where $(a',b')$ are the coordinates respectively of the new minimum and maximum created by the insertion of the soliton. Here $(\tilde{\underline a}_n,\tilde{\underline b}_n)$ denote the coordinates of the local minima and maxima that were present before the insertion of the soliton, and whose position may have changed due to the new addition. 
\begin{figure}
	
	\centering
	
	\begin{tikzpicture}[scale=0.4, thick]

    \draw[ultra thick] (0,0) -- (5,5) -- (9,1) -- (17,9) -- (21,5) -- (23,7)-- (24,6)-- (26,8) -- (34,0);

    \draw[blue, ultra thick] (22,6) -- (23,7) -- (24,6);
    
    \draw[red, ultra thick] (4,4) -- (5,5) -- (6,4);
    \draw[red, ultra thick] (16,8) -- (17,9) -- (18,8);
    \draw[red, ultra thick] (25,7) -- (26,8) -- (27,7);

\end{tikzpicture}
	
	\caption{We attached the new smallest blue soliton; the dark part of the excursion is the slot space where it is possible to attach the soliton, the red part where it is not possible. }
	
	\label{transform}
\end{figure}
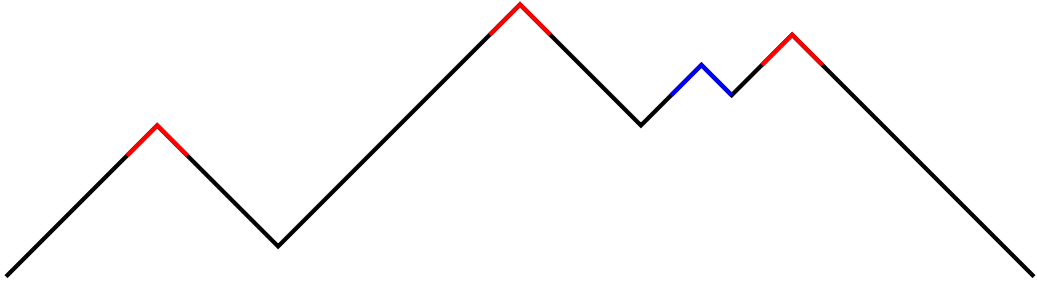

Suppose we insert the $(n+1)$-th soliton following the first $k$ maxima, and let us assume that it is attached in the left part of a soliton, so that the coordinate of the new local maximum $b'$ is smaller than the coordinate of the new local minimum $a'$.
This corresponds to the case shown in Fig.~\ref{transform} with $k=2$. The other case ($a'<b'$) can be handled in a similar way.

We have that $\tilde a_i=a_i$ and $\tilde b_i=b_i$ for $i=1,\dots , k$ so that in particular these coordinates do not depend on the variables $(u_{n+1},\ell_{n+1})$. We have moreover $b'=u_{n+1}+(2k+1)\ell_{n+1}$ and $a'=u_{n+1}+(2k+2)\ell_{n+1}$
and finally $\tilde b_i=b_i+2\ell_{n+1}$ and $\tilde a_i=a_i+2\ell_{n+1}$ for $i=k+1,\dots , n$. Therefore the matrix for the change of variables form $(\underline u_n,\underline \ell_n, u_{n+1}, \ell_{n+1})$ to $(\tilde{\underline a}_n, \tilde{\underline b}_n, a',b')$ is given as follows
\[
\left(\begin{matrix}
	A_{11}& \dots & A_{1,2n-1} & 0 & 0 \\
	\vdots & \vdots & \vdots & \vdots & \vdots \\
	A_{k-1,1} & \dots & A_{k-1,2n-1} & 0 & 0 \\
	A_{k,1} & \dots & A_{k,2n-1} & 0 & 2\\
	\vdots & \vdots & \vdots & \vdots & \vdots \\
	A_{n-1,1} & \dots & A_{n-1,2n-1} & 0 & 2 \\
	A_{n,1} & \dots & A_{n,2n-1} & 0 & 0\\
	\vdots & \vdots & \vdots & \vdots & \vdots \\
	A_{n+k-1,1} & \dots & A_{n+k-1,2n-1} & 0 & 0 \\
	A_{n+k,1} & \dots & A_{n+k, 2n-1} & 0 & 2\\
	\vdots & \vdots & \vdots & \vdots & \vdots \\
	A_{2n-1,1} & \dots & A_{2n-1,2n-1} & 0 & 2 \\
	0 & \dots & 0 & 1 & 2k+2 \\
	0 & \dots & 0 & 1 &  2k+1\\
\end{matrix}\right)
\]
If we develop the determinant of the above $(2n+1)\times (2n+1)$ matrix along the second to last column we get that, up to a $\pm$ sign in front, its determinant is equal to
\[
(2k+2)|A|-(2k+1)|A|=|A|=1\,,
\]
where the last identity follows by the inductive hypothesis.	
		
Consider $q\in \mathscr Q$. By Lemma \ref{Tpos}, $\nu_q$ given in \eqref{nuq5} is a probability measure, and relation \eqref{a=q} and the fact that the Jacobian transformation has determinant $1$ imply that $\mu_{\alpha}$ in \eqref{mua1} is a probability measure as well, so that $\alpha\in \mathcal A$. Conversely, consider $\alpha\in \mathcal A$ so that $\mu_{\alpha}$ is a probability measure, and let $q$ be computed by the solution of the Cauchy problem \eqref{lequazione1}. Once again, it follows from the fact that the determinant of the Jacobian transformation is $1$ that $\nu_q$ is a probability measure. By Remark \ref{maxi} the maximum value of the excursion has law given by \eqref{maxlaw}, and hence $q\in \mathscr Q$.
	
The bijection between the sets $\mathcal A^+$ and $\mathscr Q^+$ is similarly obtained. By Corollary \ref{mlength} the average length of the leftmost excursion $\sM$ is finite if and only if 
$q\in \mathscr Q^+$. On the other hand, from \eqref{mua2}, the excursion with law $\mu_\alpha$ has finite length only when $\alpha \in \mathcal A^+$, so we conclude that $q\in \mathscr Q^+$ if and only if $\alpha \in \mathcal A^+$.

Finally, in view of expressions \eqref{mua1} and \eqref{nuq5} of the measures $\mu_\alpha$ and $\nu_q$, respectively, and relation \eqref{a=q} that defines $\alpha$ in terms of $q$, the normalizing constants $Z(\alpha)$ and $\int_0^{+\infty} q(y)dy$ must be equal. \end{proof}

\subsubsection{Poisson intensity for the telegraph process}
\label{PoiTel}

While it is simple to compute $\alpha$ from $q$ by formula \eqref{a=q}, there is not a general form of the solution to equation \eqref{lequazione1}.
We consider here a completely solvable case, the telegraph process of Example \ref{teleg}. 

The weight function $\alpha \in \mathcal A$ is given in \eqref{zza1}, 
\begin{equation}\label{malfa}
	\alpha(\ell )=\lambda_-\lambda_+e^{-(\lambda_-+\lambda_+)\ell }.
\end{equation}

Let $A(y)\vcentcolon =\mathcal Q(y)-\frac12(\lambda_++\lambda_-)y$, where $\mathcal Q$ is the solution to the Cauchy problem \eqref{lequazione1}. In terms of $A$, this problem becomes 
\begin{equation}\label{lequazione2}
	\left\{
	\begin{array}{l}
		A''(y)=\lambda_+\lambda_- e^{2A(y)}\, , \\
		A(0)=0\, , \\
		A'(0)=-\frac12(\lambda_++\lambda_-).
	\end{array}
	\right.
\end{equation}
This is a one dimensional Newton law and in particular there is a conserved quantity given by the energy 
\begin{equation}\label{energy}
	E(A, A')\vcentcolon =(A')^2-\lambda_+\lambda_-e^{2A}\, .
\end{equation}
The constant value of the energy is given by 
\begin{equation}\label{valoreen}
	E=E(A(0), A'(0))=\frac14(\lambda_+-\lambda_-)^2\ge0\, .
\end{equation}
We consider first the case $E>0$ that corresponds to $\lambda_-<\lambda_+$. 

By \eqref{energy} we have that
\begin{equation}\label{sign1}
	A'=\pm \sqrt{E+\lambda_+\lambda_-e^{2A}}\, , 
\end{equation}
and the correct sign is the negative one to match the initial value \eqref{lequazione2}. Note that $A'<0$ at all times. Indeed, if it were to vanish, the conserved energy \eqref{energy} would become negative. We have reduced the problem to a first order equation that can be solved by separation of variables:
\begin{equation}\label{inteq}
	\int \frac{dA}{\sqrt{E+\lambda_+\lambda_-e^{2A}}}=-\int dy\, .
\end{equation}
The primitive on the left-hand side above is given by
\begin{equation}
	-\frac{1}{\sqrt E}\coth^{-1}\Big(\sqrt{1+\frac{\lambda_+\lambda_-}{E}e^{2A}}\,\Big),
\end{equation}
where $\coth(x)$ denotes the hyperbolic cotangent and $\coth^{-1}$ its inverse function,
$\coth^{-1}(x)=\frac 12 \log\frac{x+1}{x-1}$, $|x|>1$. Then \eqref{inteq} yields
\begin{equation}
	\coth^{-1}\Big(\sqrt{1+\frac{\lambda_+\lambda_-}{E}e^{2A}}\,\Big)=\sqrt E(y+K)\, , 
\end{equation}
 $K$ is a constant to be determined and from $A(0)=0$. We obtain
\begin{equation}
	K=\frac{1}{\sqrt E}\coth^{-1}\Big(\sqrt{1+\frac{\lambda_+\lambda_-}{E}}\,\Big)\, 
\end{equation}
and
\begin{align}
	A(y)=\frac 12 \log \Big\{\frac{E}{\lambda_-\lambda_+}\Big[\sinh\Big(\sqrt E(y+K)\Big)\Big]^{-2}\Big\}.
\end{align}
Since $q(y)=\mathcal Q''(y)=A''(y)=\lambda_-\lambda_+ e^{2A(y)}$ we finally get
\begin{equation}\label{maxzz}
	q(y)=\frac{E}{\Big[\sinh\Big(\sqrt E(y+K)\Big)\Big]^{2}}\, .
\end{equation}
As can be easily verified $q\in \mathscr Q^+$, which we already knew since for $\lambda_+>\lambda_-$ the excursions of the walk have finite average length. Proposition \ref{eqm9} implies that $\alpha \in \mathcal A^+$.

\smallskip
We consider now the case $E=0$ that corresponds to equal rates $\lambda_-=\lambda_+$. 
Let $\lambda\vcentcolon=\lambda_-=\lambda_+>0$. Then \eqref{sign1} reduces to $A'=-\lambda e^A$ that can be easily integrated. Using the values of the initial conditions in \eqref{lequazione2}, we obtain
\begin{equation}
	A(y)=\log\frac{1}{\lambda y+1}\quad \text{and}\quad q(y)=A''(y)=\frac{\lambda^2}{(\lambda y +1)^2}\,.
\end{equation}
Clearly in this case $q\notin \mathscr Q^+$, a fact we already knew, as the telegraph walk with equal rates is symmetric, and the mean length of the excursions is infinite.

\subsection{Random walks}
\label{ss32}

In previous sections we introduced two families of distributions on the space of excursions and proved that they are equivalent, Proposition \ref{eqm9}. We now present the extension of these measures to $\hcW$, the space of walks that have a record at the origin (see \eqref{hcW}), and discuss how these extensions are related.

We consider a walk $\xi\in \hcW$ whose excursions are iid with law $\mu_\alpha$ as in \eqref{mua1}, $\alpha \in \mathcal A$. The interval between two consecutive excursions is made up of records, that is, times at which $\xi$ achieves its value for the fist time, see \eqref{rec1}. The lengths of these record intervals are iid exponential random variables with parameter $Z(\alpha)$, the partition function in \eqref{za3}.

Recall the notation introduced in Section \ref{ipc1}. 

\begin{theorem}
	\label{pro5}
	Let $\alpha \in \mathcal A$ and $q \in \mathscr Q$ satisfy
	\begin{align}
		\label{a=q7}
		\alpha(\ell)=q(\ell)\, e^{-2\int_0^\ell(\ell-y)q(y)dy},
	\end{align}
	as in \eqref{a=q}.
	Then, 
	
	\noindent (i) If $\xi\in\hcW$ has excursions and record interval lengths $(\vep_k, \su_k)_{k\in\ZZ}$, satisfying
	\begin{gather}
		(\vep_k)_{k\in\ZZ} \text{ are iid with law $\mu_\alpha$} \label{e12}\\
		(\su_k)_{k\in\ZZ} \text{ are iid Exponential$(Z(\alpha))$} \label{a12}\\
		(\vep_k)_{k\in\ZZ} \text{ and } (\su_k)_{k\in\ZZ}\text{ are independent.}\label{i12}
	\end{gather}
	Then the slot decomposition $\sN[\xi]$ is a Poisson$(q)$ process.
	
	\noindent (ii) If $\sN$ is a Poisson$(q)$ process, then $\xi[\sN]\in\hcW$ and the excursions and record intervals of $\xi[\sN]$ satisfy \eqref{e12}, \eqref{a12}, and \eqref{i12}.
\end{theorem}
\begin{proof}
	We first show (ii). Recall $\sN_+\subset \RR_+\times\RR_+$ is a Poisson process with intensity measure $dx\, q(\ell)d\ell$. In Proposition \ref{eqm9} we have shown that, under \eqref{a=q7}, the leftmost excursion $\vep_0[\sN_+]$ has distribution $\mu_\alpha$ and that the law of $\su_0[\sN_+]$, the gap between the origin and the leftmost excursion, has law exponential$(Z(\alpha))$. To do that, we have explored the regions $\sU$ and $\sG$, whose right boundary $\sg$ is a ``hitting boundary'', in the sense that it has been computed as a function of $\sN_0\cap (\sU\cup\sG)$. This implies that, given $\sM_0=\sN_0\cap(\sU\cup\sG)$, the process $\sN_0\setminus(\sU\cup\sG)$ is a Poisson$(q)$ process restricted to $(\RR_+\times\RR_+)\setminus(\sU\cup\sG)$. The intensity measure of $\sN_+$ is homogeneous in the horizontal coordinates $x$. Then $\sN_1$, obtained in \eqref{sc111} by rigidly translating a set of points to the left, is independent of $(\su_0, \sM_0)$. This in turn implies that (a) the gap $\su_1$ between excursions $\vep_0$ and $\vep_1$ has the same law as $\su_0$, (b) the excursion $\vep_1$ has the same law as $\vep_0$ and (c) the pair $(\su_1, \vep_1)$ is independent of the pair $(\su_0, \vep_0)$. By iterating this argument we conclude that $(\su_k, \vep_k)_{k\ge0}$ are iid with the same law as $(\su_0, \vep_0)$. Since the record intervals between excursions and the excursions $(\su_k, \vep_k)_{k<0}$ have been obtained by reflection over the vertical axis, and $\sN_-$ and $\sN_+$ are independent, we get that $(\su_k, \vep_k)_{k\in\ZZ}$ are iid with the same law as $(\su_0, \vep_0)$. The variables $(\su_k)_k$ are iid with positive mean and $\sum_k \su_k=+\infty$ almost surely, so $\sN\in \cN$ (see Remark~\ref{equivTfin}). By construction \eqref{mex1}, the carrier process satisfies $\zeta[\sN](0)=0$, because $0\in[\sa_{-1}, \sa_0]$, and then $0$ is a record for $\xi$ and $\xi\in\hcW$. 
	
	Since the map $\sN\mapsto (\su_k[\sN], \vep_k[\sN])_{k\in\ZZ}$ is a bijection, item (i) follows from (ii).
\end{proof}

\section{Appendix}
\label{appendix}

We analyze several situations to clarify the soliton decomposition of excursions exhibiting ties, see Section \ref{siw3}.

When identifying the solitons, we order the runs by their lengths. For runs with the same length, we order them based on their leftmost endpoint.
The shortest run for the excursion in Fig.~\ref{ex1} is $(1,2)$.
We obtain a soliton of height 1 associated to the maximum at $1$. Then there is a soliton of height 2 corresponding to the maximum a 3. 
The tie-break rule implies that the largest soliton is associated with the rightmost point where the maximum of the excursion is reached.
This follows the same convention as in \cite{FNRW}.

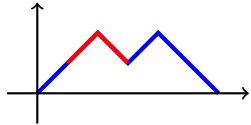
\begin{figure}[ht]
\centering
\begin{tikzpicture}[scale=0.4, thick]

\draw[blue, ultra thick] (0,0) -- (2,2) -- (3,1) -- (4,2) -- (6,0);
\draw[red, ultra thick] (1,1) -- (2,2) -- (3,1);

   \draw[->] (-1,0) -- (7,0);
   \draw[->] (0,-1) -- (0,3);
\end{tikzpicture}
\caption{Here we can see how the decomposition works when an excursion has more than one maximum at the same level. The largest soliton is associated to the maximum on the right. The leftmost maximum is associated to a shorter soliton.}
\label{ex1}
\end{figure}

In the soliton decomposition of the excursion  in Fig.~\ref{ex2} we identify the solitons in the following order: red, orange and blue.
Note that the orange and red solitons have the same height.
On the other hand, the reconstruction of the excursion from the point configuration proceeds in descending order from largest to smallest solitons. In this example we would start with the blue soliton and add the other two in any order,
the excursion obtained is the same regardless of the order in which these solitons were added.

\begin{figure}[ht]
\centering
\begin{tikzpicture}[scale=0.4, thick]

\draw[blue, ultra thick] (0,0) -- (2,2) -- (3,1)  -- (5,3) -- (7,1) -- (8,2) -- (10,0);
\draw[red, ultra thick] (1,1) -- (2,2) -- (3,1);
\draw[orange, ultra thick] (7,1) -- (8,2) -- (9,1);

   \draw[->] (-1,0) -- (11,0);
   \draw[->] (0,-1) -- (0,3);
\end{tikzpicture}
\caption{An excursion with two solitons of the same height.}
\label{ex2}
\end{figure}
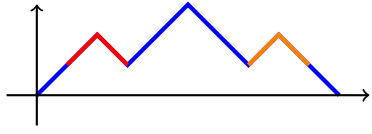

In Fig.~\ref{ex3}, the minimum of the excursion is attained at an interior point. Moreover, this minimum lies to the right of the excursion’s maximum.
In \cite{FNRW} a discrete analogue of this situation would be treated as a configuration consisting of two separate excursions.
By contrast, we consider it a single excursion.
This choice allows us to simplify the presentation, while the underlying constructions remain equivalent.

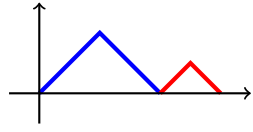
\begin{figure}[ht]
\centering
\begin{tikzpicture}[scale=0.4, thick]

\draw[blue, ultra thick] (0,0) -- (2,2) -- (4,0);
\draw[red, ultra thick] (4,0) -- (5,1) -- (6,0);

   \draw[->] (-1,0) -- (7,0);
   \draw[->] (0,-1) -- (0,3);
\end{tikzpicture}
\caption{The minimum of the excursion is reached at a point inside the excursion.}
\label{ex3}
\end{figure}

Along the article we have used the set notation for $\sN$, when in fact it is a multiset. That is, points may appear with multiplicity. See Fig.~\ref{ex4} for an example. When the points are sampled as a Poisson point process this happens with $0$ probability.

\begin{figure}[ht]
\centering
\begin{tikzpicture}[scale=0.4, thick]

\draw[blue, ultra thick] (0,0) -- (2,2) -- (4,0);
\draw[red, ultra thick] (4,0) -- (6,2) -- (8,0);

   \draw[->] (-1,0) -- (9,0);
   \draw[->] (0,-1) -- (0,3);
   
      \draw[->] (10,0) -- (13,0);
   \draw[->] (11,-1) -- (11,3);

    \fill[blue] (11,2) -- (11.2,2) arc (0:180:0.2) -- cycle;

    \fill[red] (11,2) -- (11.2,2) arc (0:-180:0.2) -- cycle;
   
\end{tikzpicture}
\caption{The point configuration has two points at (0,2).}
\label{ex4}
\end{figure}
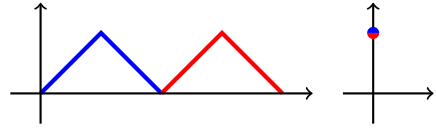

\section*{Acknowledgements}
\label{grazie}

D.G. acknowledges the financial support from the Italian Research Funding Agency (MIUR) through
PRIN project ``Emergence of condensation-like phenomena in interacting particle systems: kinetic and lattice models'', grant n. 202277WX43 and  from the Italian Ministry of Foreign Affairs and International Cooperation by the grant BR26GR05. I.A. gratefully acknowledges financial support from the Fundac\~ao de Amparo \`a Pesquisa do Estado de S\~ao Paulo (FAPESP) through project 2024/12778-0.

\section*{Data Availability} Data sharing not applicable to this article as no datasets were generated or analysed during the current study.

\section*{Conﬂicts of interest}
 The sources of ﬁnancial support received by the authors for this research are listed in
the acknowledgments section above.

\bibliographystyle{acm}
\bibliography{bbsc-bib}

@article{olla2025scalinglimitssolitonsboxball,
      title={Scaling limits of solitons in the box-ball system}, 
      author={Stefano Olla and Makiko Sasada and Hayate Suda},
      year={2025},
      journal={arXiv:2411.14818},
      eprint={2411.14818},
      archivePrefix={arXiv},
      primaryClass={math.PR},
      url={https://arxiv.org/abs/2411.14818}, 
}

@article{PitmanBM-Bessel,
 ISSN = {00018678},
 URL = {http://www.jstor.org/stable/1426125},
 author = {James W. Pitman},
 journal = {Advances in Applied Probability},
 number = {3},
 pages = {511--526},
 publisher = {Applied Probability Trust},
 title = {One-Dimensional {Brownian} Motion and the Three-Dimensional {Bessel} Process},
 urldate = {2026-06-17},
 volume = {7},
 year = {1975}
}

@book {evans,
    AUTHOR = {Evans, Steven N.},
     TITLE = {Probability and real trees},
    SERIES = {Lecture Notes in Mathematics},
    VOLUME = {1920},
      NOTE = {Lectures from the 35th Summer School on Probability Theory
              held in Saint-Flour, July 6--23, 2005},
 PUBLISHER = {Springer, Berlin},
      YEAR = {2008},
     PAGES = {xii+193},
      ISBN = {978-3-540-74797-0},
   MRCLASS = {60B99 (05C05 60J25)},
  MRNUMBER = {2351587},
MRREVIEWER = {Wolfgang Woess},
       DOI = {10.1007/978-3-540-74798-7},
       URL = {https://doi.org/10.1007/978-3-540-74798-7},
}

@article{MIT,
  title={The exact correspondence between conserved quantities of a periodic box-ball system and string solutions of the {Bethe} ansatz equations},
  author={Mada, Jun and Idzumi, Makoto and Tokihiro, Tetsuji},
  journal={Journal of mathematical physics},
  volume=47,
  number=5,
  pages=053507,
  year=2006,
  publisher={AIP}
}

@article{cbs,
  title={The {GGE} averaged currents of the classical {T}oda chain},
  author={Cao, Xiangyu and Bulchandani, Vir B and Spohn, Herbert},
  journal={Journal of Physics A: Mathematical and Theoretical},
  volume={52},
  number={49},
  pages={495003},
  year={2019},
  publisher={IOP Publishing}
}

@article{TS90,
author = {Takahashi ,Daisuke and Satsuma ,Junkichi},
title = {A Soliton Cellular Automaton},
journal = {Journal of the Physical Society of Japan},
volume = 59,
number = 10,
pages = {3514-3519},
year = 1990,
doi = {10.1143/JPSJ.59.3514},
URL = { https://doi.org/10.1143/JPSJ.59.3514},
eprint = {https://doi.org/10.1143/JPSJ.59.3514}
}

@article{bulchandani-cao-moore,
   title={Kinetic theory of quantum and classical {T}oda lattices},
   volume={52},
   ISSN={1751-8121},
   url={http://dx.doi.org/10.1088/1751-8121/ab2cf0},
   DOI={10.1088/1751-8121/ab2cf0},
   number={33},
   journal={Journal of Physics A: Mathematical and Theoretical},
   publisher={IOP Publishing},
   author={Bulchandani, Vir B and Cao, Xiangyu and Moore, Joel E},
   year={2019},
   month={Jul},
   pages={33LT01}
}

@article{FNRW,
  
 Author = {Ferrari, Pablo A. and Nguyen, Chi and Rolla, Leonardo T. and Wang, Minmin},
 Title = {Soliton decomposition of the box-ball system},
 FJournal = {Forum of Mathematics, Sigma},
 Journal = {Forum Math. Sigma},
 ISSN = {2050-5094},
 Volume = {9},
 Pages = {37},
 Note = {Id/No e60},
 Year = {2021},
 Language = {English},
 DOI = {10.1017/fms.2021.49}
}

@Article{KOSTY2006,
 Author = {Kuniba, Atsuo and Okado, Masato and Sakamoto, Reiho and Takagi, Taichiro and Yamada, Yasuhiko},
 Title = {Crystal interpretation of {Kerov}-{Kirillov}-{Reshetikhin} bijection},
 FJournal = {Nuclear Physics. B},
 Journal = {Nucl. Phys., B},
 ISSN = {0550-3213},
 Volume = {740},
 Number = {3},
 Pages = {299--327},
 Year = {2006},
 Language = {English},
 DOI = {10.1016/j.nuclphysb.2006.02.005},
 zbMATH = {5145271},
 Zbl = {1109.81043}
}

@article{Doyon_2019,
   title={Generalized hydrodynamics of the classical {T}oda system},
   volume={60},
   ISSN={1089-7658},
   url={http://dx.doi.org/10.1063/1.5096892},
   DOI={10.1063/1.5096892},
   number={7},
   journal={Journal of Mathematical Physics},
   publisher={AIP Publishing},
   author={Doyon, Benjamin},
   year={2019},
   month={Jul},
   pages={073302}
}

@article{Spohn_2019,
   title={Generalized {G}ibbs {E}nsembles of the Classical {T}oda Chain},
   ISSN={1572-9613},
   url={http://dx.doi.org/10.1007/s10955-019-02320-5},
   DOI={10.1007/s10955-019-02320-5},
   journal={Journal of Statistical Physics},
   publisher={Springer Science and Business Media LLC},
   author={Spohn, Herbert},
   year={2019},
   month={May}
}

@InCollection{FG18,
 Author = {Ferrari, Pablo A. and Gabrielli, Davide},
 Title = {Box-ball system: soliton and tree decomposition of excursions},
 BookTitle = {{XIII} symposium on probability and stochastic processes. Contributions and lecture notes, {U}niversidad {N}acional {A}ut\'onoma de {M}\'exico, {UNAM}, {M}\'exico {C}ity, {M}exico, {D}ecember 4--8, 2017},
 ISBN = {978-3-030-57512-0; 978-3-030-57513-7},
 Pages = {107--152},
 Year = {2020},
 Publisher = {Cham: Birkh{\"a}user},
 Language = {English},
 DOI = {10.1007/978-3-030-57513-7_5},
 zbMATH = {7374180}
}

@Article{FG19,
 Author = {Ferrari, Pablo A. and Gabrielli, Davide},
 Title = {{BBS} invariant measures with independent soliton components},
 FJournal = {Electronic Journal of Probability},
 Journal = {Electron. J. Probab.},
 ISSN = {1083-6489},
 Volume = {25},
 Pages = {26},
 Note = {Id/No 78},
 Year = {2020},
 Language = {English},
 DOI = {10.1214/20-EJP475},
 zbMATH = {7252710}
}

@Article{HamblyMartinOConnell01,
  author  = {Hambly, Ben and Martin, James B. and O'Connell, N.},
  title   = {Pitman's {$2M-X$} Theorem for skip-free random walks with {M}arkovian increments},
  journal = {Electron. Commun. Probab.},
  year    = {2001},
  volume  = {6},
  pages   = {73--77},
  doi     = {10.1214/ECP.v6-1036},
}

@article {CroydonSasada20,
    AUTHOR = {Croydon, David A. and Sasada, Makiko},
     TITLE = {Generalized hydrodynamic limit for the box-ball system},
   JOURNAL = {Comm. Math. Phys.},
  FJOURNAL = {Communications in Mathematical Physics},
    VOLUME = {383},
      YEAR = {2021},
    NUMBER = {1},
     PAGES = {427--463},
      ISSN = {0010-3616},
   MRCLASS = {82C22 (37K40 82C23 82C70)},
  MRNUMBER = {4236070},
       DOI = {10.1007/s00220-020-03914-x},
       URL = {https://doi.org/10.1007/s00220-020-03914-x},
}

@InCollection{zbMATH07500585,
 Author = {Croydon, David A. and Sasada, Makiko},
 Title = {Discrete integrable systems and {Pitman}'s transformation},
 BookTitle = {Stochastic analysis, random fields and integrable probability -- Fukuoka 2019. Proceedings of the 12th Mathematical Society of Japan, Seasonal Institute (MSJ-SI), Kyushu University, Japan, 31 July -- 9 August 2019},
 ISBN = {978-4-86497-094-5},
 Pages = {381--403},
 Year = {2021},
 Publisher = {Tokyo: Mathematical Society of Japan},
 Language = {English},
 DOI = {10.2969/aspm/08710381},
 zbMATH = {7500585}
}

@article{zbMATH07620722,
 author = {Croydon, David A. and Sasada, Makiko and Tsujimoto, Satoshi},
 title = {Bi-infinite solutions for {KdV}- and {Toda}-type discrete integrable systems based on path encodings},
 fjournal = {Mathematical Physics, Analysis and Geometry},
 journal = {Math. Phys. Anal. Geom.},
 issn = {1385-0172},
 volume = {25},
 number = {4},
 pages = {71},
 note = {Id/No 27},
 year = {2022},
 language = {English},
 doi = {10.1007/s11040-022-09435-4},
 zbMATH = {7620722},
 Zbl = {1508.37098}
}

@book{zbMATH07653279,
 author = {Croydon, David A. and Kato, Tsuyoshi and Sasada, Makiko and Tsujimoto, Satoshi},
 title = {Dynamics of the box-ball system with random initial conditions via {Pitman}'s transformation},
 fseries = {Memoirs of the American Mathematical Society},
 series = {Mem. Am. Math. Soc.},
 issn = {0065-9266},
 volume = {1398},
 isbn = {978-1-4704-5633-7; 978-1-4704-7399-0},
 year = {2023},
 publisher = {Providence, RI: American Mathematical Society (AMS)},
 language = {English},
 doi = {10.1090/memo/1398},
 zbMATH = {7653279},
 Zbl = {1510.37002}
}

@article{ffgs,
  title={Hard Rod Hydrodynamics and the {Levy} {Chentsov} Field},
  author={Ferrari, Pablo A and Franceschini, Chiara and Grevino, Dante GE and Spohn, Herbert},
  journal={arXiv preprint arXiv:2211.11117},
  year={2022}
}

@article{spohn2023,
  title={Hydrodynamic scales of integrable many-particle systems},
  author={Spohn, Herbert},
  journal={arXiv preprint arXiv:2301.08504},
  year={2023}
}

@article{fo2022,
  title={Diffusive Fluctuations in Hard Rods System},
  author={Olla, Stefano and Ferrari, Pablo A.},
  journal={arXiv preprint arXiv:2210.02079},
  year={2022}
}

@Article{cs-markov,
 Author = {Croydon, David A. and Sasada, Makiko},
 Title = {Invariant measures for the box-ball system based on stationary {Markov} chains and periodic {Gibbs} measures},
 FJournal = {Journal of Mathematical Physics},
 Journal = {J. Math. Phys.},
 ISSN = {0022-2488},
 Volume = {60},
 Number = {8},
 Pages = {083301, 25},
 Year = {2019},
 Language = {English},
 DOI = {10.1063/1.5095622},
 zbMATH = {7102716},
 Zbl = {1426.37013}
}

@Article{BS,
 Author = {Boldrighini, C. and Suhov, Y. M.},
 Title = {One-dimensional hard-rod caricature of hydrodynamics: ``{Navier}-{Stokes} correction'' for local equilibrium initial states},
 FJournal = {Communications in Mathematical Physics},
 Journal = {Commun. Math. Phys.},
 ISSN = {0010-3616},
 Volume = {189},
 Number = {2},
 Pages = {577--590},
 Year = {1997},
 Language = {English},
 DOI = {10.1007/s002200050218},
 zbMATH = {1095963},
 Zbl = {0895.76002}
}

@article{bds,
   author={Boldrighini, C. and Dobrushin, R. L. and Sukhov, Yu. M.},
   title={One-dimensional hard rod caricature of hydrodynamics},
   volume={31},
   number={3},
   journal={J. Stat. Phys.},
   pages={577-616},
   year={1983},
}

@article{mss2023,
 author = {Mucciconi, Matteo and Sasada, Makiko and Sasamoto, Tomohiro and Suda, Hayate},
 title = {Relationships between two linearizations of the box-ball system: {Kerov}-{Kirillov}-{Reschetikhin} bijection and slot configuration},
 fjournal = {Forum of Mathematics, Sigma},
 journal = {Forum Math. Sigma},
 issn = {2050-5094},
 volume = 12,
 pages = 40,
 note = {Id/No e55},
 year = 2024,
 language = {English},
 doi = {10.1017/fms.2024.39},
 zbMATH = 7846243,
 Zbl = {1546.82043}
}

@Article{KS2009,
 Author = {Kirillov, Anatol N. and Sakamoto, Reiho},
 Title = {Relationships between two approaches: rigged configurations and 10-eliminations},
 FJournal = {Letters in Mathematical Physics},
 Journal = {Lett. Math. Phys.},
 ISSN = {0377-9017},
 Volume = 89,
 Number = 1,
 Pages = {51--65},
 Year = 2009,
 Language = {English},
 DOI = {10.1007/s11005-009-0318-3},
 zbMATH = 5601503,
 Zbl = {1251.37019}
}

@article{zbMATH03491985,
 author = {Kac, Mark},
 title = {A stochastic model related to the telegrapher's equation},
 fjournal = {Rocky Mountain Journal of Mathematics},
 journal = {Rocky Mt. J. Math.},
 issn = {0035-7596},
 volume = {4},
 pages = {497--509},
 year = {1974},
 language = {English},
 doi = {10.1216/RMJ-1974-4-3-497},
 zbMATH = {3491985},
 Zbl = {0314.60052}
}

@book{zbMATH07582508,
 author = {Ratanov, Nikita and Kolesnik, Alexander D.},
 title = {Telegraph processes and option pricing},
 edition = {2nd edition},
 isbn = {978-3-662-65826-0; 978-3-662-65827-7},
 year = {2022},
 publisher = {Cham: Springer},
 language = {English},
 doi = {10.1007/978-3-662-65827-7},
 zbMATH = {7582508},
 Zbl = {1515.60017}
}

@misc{arXiv:2604.14346,
 author = {Amol Aggarwal and Matthew Nicoletti},
 title = {Fluctuations for the {Toda} lattice},
 year = {2026},
 howpublished = {Preprint, {arXiv}:2604.14346 [math.{PR}] (2026)},
 url = {https://arxiv.org/abs/2604.14346},
 arXiv = {arXiv:2604.14346}
}

@book{zbMATH07852571,
 author = {Spohn, Herbert},
 title = {Hydrodynamic scales of integrable many-body systems},
 isbn = {978-981-12-8352-9; 978-981-12-8354-3},
 year = {2024},
 publisher = {Singapore: World Scientific},
 language = {English},
 doi = {10.1142/13600},
 zbMATH = {7852571},
 Zbl = {1553.37019}
}

\vspace{2cm}
\parskip 0pt
\parindent 0pt

\noindent Inés Armendáriz, Pablo Blanc, Pablo A. Ferrari

\noindent {Departamento de Matemática, Facultad de Ciencias Exactas y Naturales, Universidad de Buenos Aires and IMAS-CONICET, Buenos Aires, Argentina}

\noindent

\href{mailto:iarmend@dm.uba.ar}{\tt iarmend@dm.uba.ar}, 
\href{mailto:pblanc@dm.uba.ar}{\tt pblanc@dm.uba.ar}, 
\href{mailto:pferrari@dm.uba.ar}{\tt pferrari@dm.uba.ar}, 

\href{https://mate.dm.uba.ar/~pferrari/}{mate.dm.uba.ar/$\sim$pferrari/}

\bigskip
\noindent Davide Gabrielli

\noindent {Department of Mathematics, University of L'Aquila,  L'Aquila, Italy}

\noindent \href{mailto: gabriell@univac.it}{\tt gabriell@univac.it}

\href{https://people.disim.univaq.it/~gabriell/}{people.disim.univaq.it/$\sim$gabriell/}

\end{document}